\documentclass[12pt]{amsart}
\usepackage[utf8]{inputenc}
\usepackage[margin=1in]{geometry}  
\usepackage{verbatim}
\usepackage{graphicx}              
\usepackage{amsmath,amsfonts,amsthm,amssymb}
\usepackage{pifont}
\usepackage{hyperref}
\usepackage{tikz} 
 \usepackage{relsize}
\usetikzlibrary{shapes}

\usetikzlibrary{positioning}
\usepackage{ytableau,enumitem}   
\usepackage{xcolor}
\usepackage{float}
\usepackage{pifont}
\newtheorem*{theorem0}{Theorem 4.7}

\usepackage{subcaption}
\newtheorem{thm}{Theorem}[section]
\newtheorem{lemma}[thm]{Lemma}

\newtheorem{prop}[thm]{Proposition}
\newtheorem{corollary}[thm]{Corollary}

\definecolor{ultragreen}{RGB}{24,120,50} 
\newtheorem{example}[thm]{\color{ultragreen}Example}
\newtheorem{conj}[thm]{Conjecture}
\newtheorem{defn}[thm]{Definition}
\newtheorem{rmk}[thm]{Remark}
\newtheorem{obs}[thm]{Observation}
\newtheorem{thmn}{Theorem}
\makeatletter
\newtheorem*{rep@theorem}{\rep@title}
\newcommand{\newreptheorem}[2]{%
\newenvironment{rep#1}[1]{%
 \def\rep@title{#2 \ref{##1}$'$}%
 \begin{rep@theorem}}%
 {\end{rep@theorem}}}
\makeatother
\newreptheorem{theorem}{Theorem}

\newcommand{\qbinom}{\genfrac{[}{]}{0pt}{}}
\newcommand{\Vol}{\mathbf{Vol}}
\newcommand{\target}{\color{red}\mathrel{\bigcirc\hspace{-2.9mm}\bullet}}

\makeatletter
\newcommand{\dicircle}{\mathbin{\mathpalette\make@circled\searrow}}
\newcommand{\make@circled}[2]{%
  \ooalign{$\m@th#1\smallbigcirc{#1}$\cr\hidewidth$\m@th#1#2$\hidewidth\cr}%
}
\newcommand{\smallbigcirc}[1]{%
  \vcenter{\hbox{\scalebox{1.4}{$\m@th#1\bigcirc$}}}%
}
\makeatother

\numberwithin{thm}{section}
\newcommand{\close}{\circlearrowright\! }

\newcommand{\CLP}{\mathrm{P_{CL}}}
\newcommand{\CL}{\mathrm{CL}}
\newcommand{\HL}{\mathrm{CL}}
\newcommand{\rmm}{\mathfrak{M}_q}
\newcommand{\drm}{\mathfrak{M}_q}
\newcommand{\rank}{\mathfrak{R}}
\definecolor{coral}{RGB}{245, 93, 159}
\definecolor{magenta}{RGB}{164, 93, 245}
\definecolor{teal}{RGB}{28, 157, 186}
\definecolor{pea}{RGB}{118, 186, 28}
\newcommand{\definition}[1]{{\color{coral}\emph{#1}}}
\usepackage[colorinlistoftodos]{todonotes}
\newcommand{\ezgi}[1]{\todo[inline,color=purple!30]{#1 \\ \hfill --- Ezgi}}
\newcommand{\yalim}[1]{\todo[inline,color=green!30]{#1 \\ \hfill --- Yalım}}

\title{Chainlink Polytopes and Ehrhart-Equivalence}

\author[E. Kantarcı Oğuz]{Ezgi Kantarcı Oğuz}
\address{Galatasaray University}\thanks{EKO was partially supported by Tübitak BİDEP 2218-121C385.}
\email{ezgikantarcioguz@gmail.com}
\author[C. Y. Özel]{Cem Yalım Özel}
\address{Boğaziçi University}
\email{yalim98@gmail.com}
\author[M. Ravichandran]{Mohan Ravichandran}
\address{Boğaziçi University}
\email{mohan.ravichandran@gmail.com}
\thanks{MR gratefully acknowledges financial support from the Bogazici Solidarity fund.}

\begin{document}
\maketitle

\begin{abstract}
We introduce a class of polytopes that we call chainlink polytopes and which allow us to construct infinite families of pairs of non isomorphic rational polytopes with the same Ehrhart quasi-polynomial. Our construction is based on circular fence posets, which admit a non-obvious and non-trivial symmetry in their rank sequences that turns out to be reflected in the polytope level. We introduce the related class of chainlink posets and show that they exhibit the same symmetry properties. We further prove an outstanding conjecture on the unimodality of circular rank polynomials.
\end{abstract}

\section{Introduction}

This paper is about a class of polytopes, that naturally arise in poset theory, specifically in the study of fence posets and related objects. They are easy to describe, pliable to study, possess certain unexpected properties and throw up several puzzles.These polytopes will be indexed by compositions; let $\bar{a} = (a_1, \ldots, a_s)$ be a composition of $n$ and let $l$ be a non-negative integer. The \definition{chainlink polytope} $\CL(\bar{a},l)$ with \definition{chain composition} $\bar{a}$ and \definition{link number} $l$ is defined to be the polytope:
\[\CL(\bar{a},l)= \{x \in \mathbb{R}^s \mid 0 \leq x_i \leq a_i,  \, x_{i} - x_{i+1 \,(\operatorname{mod} s)} \leq a_{i}-l, \, i \in [s]\}.\]
This is a polytope that naturally lies in $\mathbb{R}^s$ and has a maximum of $3s$ facets. When the link number $l$ is equal to zero, the second set of constraints become redundant and the polytope becomes a cuboid,
\[\CL(\bar{a},0)= [0, a_1] \times [0, a_2] \times \ldots \times [0, a_s].\]
When the link number is larger, new facets emerge. For an example, see Figure \ref{fig:ChainlinkFirst}.

\begin{figure}[ht]
\includegraphics[width = 0.6\linewidth]{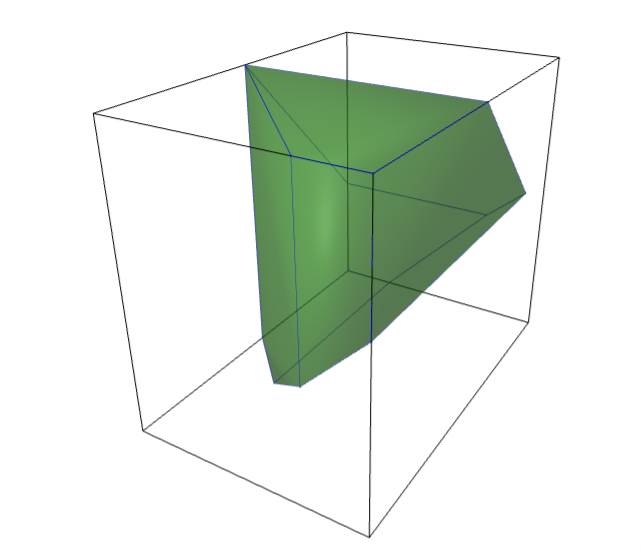}
\caption{The polytope $\CL(\bar{a}= (6, 4, 5), \, l = 2)$}
\label{fig:ChainlinkFirst}
\end{figure}

We will also work with certain special sections of these chainlink polytopes. For a positive real number $t$, we define \[\HL^t(\bar{a},l)=\HL(\bar{a},l) \cap \{x_1 + \ldots + x_s = t\}.\]

The polytopes $CL^t(\bar{a}, l)$ are non-empty for $t \in [0, n]$, where $n = a_1 + \ldots + a_s$. One of the main results in this paper is the following (unexpected) symmetry property of these sections of chainlink polytopes.

\begin{theorem0} \label{thm:vol} Let $\bar{a} = (a_1, \ldots, a_s)$ be a composition of $n$, let $l$ be a positive integer such that $2l \leq \operatorname{min} \{a_i\}_{i \in [s]}$ and let $t$ be a positive integer. Then complementary sections of the chainlink polytope have the same volume,
 \[|\HL^t(\bar{a},l)| = |\HL^{n-t}(\bar{a},l)|,\]
 where $|P|$ for a polytope denotes the relative volume. 
\end{theorem0}
There will be no ambiguity in the definition of the relative volume for us. All our polytopes will lie on hyperplanes of the form $\{x_1 + \ldots + x_s = t\}$ and we will work with the volume form that assigns volume $1$ to the polytope $\mathcal{P} = \operatorname{conv}\{0, e_1 - e_2, e_1 - e_3, \ldots, e_1 - e_s\}$.

This theorem is a special case of the following more general theorem. The terms used will be formally defined in the next section. 

\begin{thmn}  Let $\bar{a} = (a_1, \ldots, a_s)$ be a composition of $n$, let $l$ be a non-negative integer such that $2l \leq \operatorname{min} \{a_i\}_{i \in [s]}$ and let $t$ be a positive integer. Then complementary sections of the chainlink polytope have the same Ehrhart quasipolynomial,
 \[\operatorname{Ehr}\, \HL^t(\bar{a},l) = \,\operatorname{Ehr}\,\,\HL^{n-t}(\bar{a},l).\]
\end{thmn}

To see why this is unexpected, consider the chainlink polytope $CL((6, 4, 5), 2)$ as above. The sections at $t = 4$ and $t = 11$ have the same volume, but are non-isomorphic. We plot the first one on the triangular lattice, the natural choice given that these lie on the hyperplanes $x + y + z = \operatorname{const}$. 

\vspace{0.1in}
\begin{center}
\begin{figure}[ht]
\usetikzlibrary{calc}
\newcommand*\rows{4}
\begin{tikzpicture}
    \foreach \row in {0, 1, ...,4} {
        \draw[gray] ($(\row*.5, {\row*0.5*sqrt(3)})$) -- ($(4-\row*.5,{ \row*0.5*sqrt(3)})$);
        \draw[gray] ($\row*(1, 0)$) -- ($(\rows/2,{\rows/2*sqrt(3)})+\row*(0.5,{-0.5*sqrt(3)})$);
        \draw[gray] ($\row*(1, 0)$) -- ($(0,0)+\row*(0.5,{0.5*sqrt(3)})$);
    }
        \fill[red] (1,0)  circle[radius=2pt];
    \fill[red] (3,0)  circle[radius=2pt];
    \fill[red] (3.75,0.43)  circle[radius=2pt];
    \fill[red] (2,3.46)  circle[radius=2pt];
    \fill[red] (1,1.73)  circle[radius=2pt];

    \draw[red, very thick] (1,0) -- (3, 0) -- (3.75, 0.43) -- (2, 3.46) -- (1, 1.73) -- (1, 0);
\end{tikzpicture} \quad \quad
\begin{tikzpicture}
    \foreach \row in {0, 1, ...,4} {
        \draw[gray] ($(2.75*2.18181818-.88+\row*.5, -{\row*0.5*sqrt(3)})$) -- ($(2.75*2.18181818-.88+4-\row*.5,-{ \row*0.5*sqrt(3)})$);
        \draw[gray] ($(\row+2.75*2.18181818-.88, 0)$) -- ($(+2.75*2.18181818-.88+\rows/2,-{\rows/2*sqrt(3)})+\row*(0.5,-{-0.5*sqrt(3)})$);
        \draw[gray] ($(\row+2.75*2.18181818-.88, 0)$) -- ($(2.75*2.18181818-.88,0)+\row*(0.5,-{0.5*sqrt(3)})$);
    }
    \fill[blue] ($(2.75*2.18181818+1.12, -{2*sqrt(3)})$) circle[radius=2pt];
    \fill[blue] ($(2.75*2.72727273+1.12,{2.75*1.57272727-1.71-2*sqrt(3)})$) circle[radius=2pt];
    \fill[blue] ($(2.75*1.63636364+1.12, {2.75*1.57272727-1.71-2*sqrt(3)})$) circle[radius=2pt];
    \fill[blue] ($(2.75*1.63636364+1.12, {2.75*1.88727273-1.71-2*sqrt(3)})$) circle[radius=2pt];
    \fill[blue] ($(2.7*2.18181818+1.12, {2.75*1.88727273-1.71-2*sqrt(3)})$) circle[radius=2pt];

    \draw[blue, very thick] ($(2.75*2.18181818+1.12, -{2*sqrt(3)})$)--
   ($(2.75*2.72727273+1.12,{2.75*1.57272727-1.71-2*sqrt(3)})$)-- 
 ($(2.7*2.18181818+1.12, {2.75*1.88727273-1.71-2*sqrt(3)})$)--
 ($(2.75*1.63636364+1.12, {2.75*1.88727273-1.71-2*sqrt(3)})$)--
($(2.75*1.63636364+1.12, {2.75*1.57272727-1.71-2*sqrt(3)})$)--
($(2.75*2.18181818+1.12, -{2*sqrt(3)})$);
\end{tikzpicture}
\caption{The sections \color{red}$CL^4(\bar{a}, l)$ \color{black}  
 and \color{blue} $CL^{11}(\bar{a}, l)$\color{black} where $\bar{a} = (6, 4, 5)$ and $l = 2$.}
\end{figure}
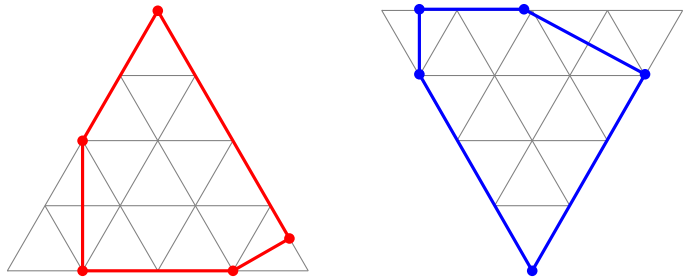
\end{center}

\begin{rmk} The theorem may fail when $2l > \operatorname{min} \{a_i\}_{i \in [s]}$. For instance, take our running example $\bar{a} = (6, 4, 5)$, but with $l = 3$ rather than $2$. Let $t = 7$ with complementary section corresponding to $n - t = 15 - 7 = 8$. In this case, we have that even the number of lattice points (the Ehrhart quasi polynomials evaluated at $1$) are different, 
\[\# \CL^t(\bar{a}, l) = 9, \qquad \# \CL^{n - t}(\bar{a}, l) = 10,\]
as are the volumes. We have that 
\[| \CL^t(\bar{a}, l)| = \dfrac{71}{12} \sqrt{3}, \qquad |\CL^{n - t}(\bar{a}, l)| = 6\sqrt{3}.\]

\end{rmk}

We do not have a conceptual explanation for why we may lose symmetry when $2l > \operatorname{min} \{a_i\}_{i \in [s]}$. However, in Section \ref{sec:3}, we will closely relate these polytopes with a natural class of posets, the so called \emph{circular fence posets} and this perspective will provide some insight into this phenomenon. 

\section{Background}
Where do these chainlink polytopes come from? At first sight, they (might perhaps) seem unmotivated, if (again perhaps) natural. We were led to these following the paper by the first and third authors \cite{main} on fence posets and in particular, a tricky problem that they had been unable to solve. We first recall the definition of fence posets. 
\begin{defn}
Given a composition $\bar{c} = (c_1, \ldots, c_k)$, the \definition{fence poset} is the poset on $n+1$ nodes, where $n = c_1 + \ldots + c_k + 1$ defined by the cover relations 
\[x_1 \prec x_2 \ldots \prec x_{c_1+1} \succ x_{c_1 + 2} \succ \ldots \succ x_{c_1 + c_2 + 1} \prec x_{c_1 + c_2 + 2} \prec \ldots.\]
\end{defn}
These posets arise in a number of contexts including cluster algebras, quiver representation theory and combinatorics. They also appeared in recent work of Morier-Genoud and Ovisenko \cite{morier2020continued}, where they introduced and studied a $q-$deformation of the rational numbers. In this same paper, the authors conjectured the following, which was proved by the first two authors.
\begin{thm}\label{thm:1}
The rank polynomials of fence posets are unimodal. 
\end{thm}
The main step the proof of this theorem involved the introduction of an ancillary class of posets, the so called \emph{circular fence posets}, and an unexpected property of these posets. 
\begin{defn}
Given an even length composition $\bar{c} = (c_1, \ldots, c_{2s})$, the circular fence poset $\bar{\mathcal{F}}(\bar{c})$ is the poset on $n$ nodes where $n = c_1 + \ldots + c_{2s}$, defined by the cover relations 
\[x_1 \prec x_2 \ldots \prec x_{c_1+1} \succ x_{c_1 + 2} \succ \ldots \succ x_{c_1 + c_2 + 1} \prec x_{c_1 + c_2 + 2} \prec \ldots \prec x_{1+\sum_{1}^{2s-1} c_i}  \succ  \ldots \succ x_{\sum_{1}^{2s} c_i} \succ x_1.\]
In other words, this is what we get by identifying the two end points a regular fence poset. 
\end{defn}

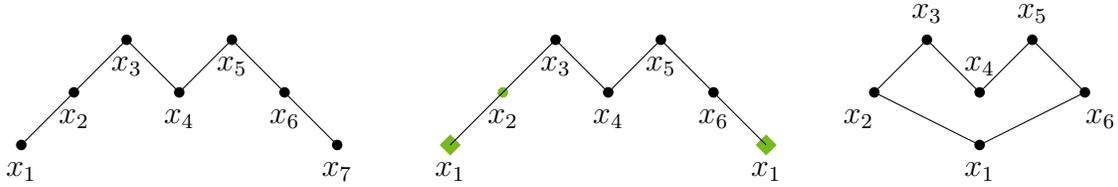
\begin{figure}[ht]
\begin{tikzpicture}[scale=.7]
\fill(0,0) circle(.1);
\fill(1,1) circle(.1);
\fill(2,2) circle(.1);
\fill(3,1) circle(.1);
\fill(4,2) circle(.1);
\fill(5,1) circle(.1);
\fill(6,0) circle(.1);
\draw (0,0)--(2,2)--(3,1)--(4,2)--(6,0);
\draw (0,-.5) node{$x_1$};
\draw (1,0.5) node{$x_2$};
\draw (2,1.5) node{$x_3$};
\draw (3,0.5) node{$x_4$};
\draw (4,1.5) node{$x_5$};
\draw (5,0.5) node{$x_6$};
\draw (6,-0.5) node{$x_7$};
\end{tikzpicture}\qquad 
\begin{tikzpicture}[scale=.7]
\fill[pea](-.2,0)--(0,-.2)--(.2,0)--(0,.2)--(-.2,0);
\fill[pea](1,1) circle(.1);
\fill(2,2) circle(.1);
\fill(3,1) circle(.1);
\fill(4,2) circle(.1);
\fill(5,1) circle(.1);
\fill[pea] (5.8,0)--(6,-.2)--(6.2,0)--(6,.2)--(5.8,0);
\draw (0,0)--(2,2)--(3,1)--(4,2)--(6,0);
\draw (0,-.5) node{$x_1$};
\draw (1,0.5) node{$x_2$};
\draw (2,1.5) node{$x_3$};
\draw (3,0.5) node{$x_4$};
\draw (4,1.5) node{$x_5$};
\draw (5,0.5) node{$x_6$};
\draw (6,-0.5) node{$x_1$};
\end{tikzpicture} 
\quad 
\begin{tikzpicture}[scale=.7]
\fill(1,1) circle(.1);
\fill(2,2) circle(.1);
\fill(3,1) circle(.1);
\fill(4,2) circle(.1);
\fill(5,1) circle(.1);
\fill(3,0) circle(.1);

\draw (3, 0)--(1, 1)--(2,2)--(3,1)--(4,2)--(5, 1)--(3, 0);
\draw (0.7,0.5) node{$x_2$};
\draw (2,2.5) node{$x_3$};
\draw (3,1.5) node{$x_4$};
\draw (4,2.5) node{$x_5$};
\draw (5.3,0.5) node{$x_6$};
\draw (3,-0.5) node{$x_1$};

\end{tikzpicture} 
\caption{The fence poset $F(2,1,1,2)$ (left) and two depictions of the circular fence poset $\bar{F}(2, 1, 1, 2)$(center and right). In the middle one, the two nodes marked $x_1$ are identified.}\label{fig:2112}
\end{figure}

In \cite{main}, the authors showed that circular fence posets satisfy an apriori unexpected property. 
\begin{thm}[\cite{main}]\label{thm:sym}
Rank polynomials of circular fence posets are symmetric. 
\end{thm}

Let us make a comment on why this result is unexpected. Given a composition $\bar{c} = (c_1, \ldots, c_{2s})$, let $\operatorname{shift}\bar{c}$ be the composition that is the cyclical shift of $\bar{c}$, that is $\operatorname{shift}\bar{c} = (c_{2s}, c_1, c_2, \ldots, c_{2s-1})$. A calculation shows that the symmetry of the rank polynomials of $\bar{F}(\bar{c})$ is equivalent to the statement that the posets $\bar{F}(\bar{c})$ and $\bar{F}(\operatorname{shift}\bar{c})$ have the same rank polynomial. It is also possible to see that this same rank symmetry may also be expressed as saying that the poset of lower ideals (our $\bar{F}(\bar{a})$) and the poset of upper ideals of the same fence poset have the same rank polynomal. However, except under very special cases, the two posets \emph{are not isomorphic}. For instance, take $\bar{c} = (2, 1, 1, 2)$ as in Figure \ref{fig:2112}. We have the following Hasse diagrams. The second fence can be seen as the vertical reflection of the first and we have labeled the elements appropriately. 

\begin{table}
\begin{tabular}{|c|c|c|}
\hline
Composition&Fence Poset & Hasse Diagram of Lattice of Lower Ideals\\
\hline 
\raisebox{2cm}{\begin{tabular}{c} $\bar{c}$\\ \\$(2, 1, 1, 2)$\end{tabular}} & \raisebox{.8cm}{\begin{tikzpicture}[scale=.7]
\fill(1,1) circle(.1);
\fill(2,2) circle(.1);
\fill(3,1) circle(.1);
\fill(4,2) circle(.1);
\fill(5,1) circle(.1);
\draw (0,0)--(2,2)--(3,1)--(4,2)--(6,0);
\draw (0,-.5) node{$1$};
\draw (1,0.5) node{$2$};
\draw (2,1.5) node{$3$};
\draw (3,0.5) node{$4$};
\draw (4,1.5) node{$5$};
\draw (5,0.5) node{$6$};
\draw (6,-0.5) node{$1$};
\fill[pea](5.8,0)--(6,-.2)--(6.2,0)--(6,.2)--(5.8,0);
\fill[pea](-.2,0)--(0,-.2)--(.2,0)--(0,.2)--(-.2,0);
\end{tikzpicture} }&
\begin{tikzpicture}[scale=1.1]
\fill(0,0) circle(.1);
\fill(1,1) circle(.1);
\fill(1,-1) circle(.1);
\fill(2,1) circle(.1);
\fill(2,0) circle(.1);
\fill(2,-1) circle(.1);
\fill(3,1) circle(.1);
\fill(3,0) circle(.1);
\fill(3,-1) circle(.1);
\fill(4,1) circle(.1);
\fill(4,0) circle(.1);
\fill(4,-1) circle(.1);
\fill(5,1) circle(.1);
\fill(5,-1) circle(.1);
\fill(6,0) circle(.1);
\draw (0,0)--(1,1);
\draw (0,0)--(1,-1);
\draw (1,1)--(2,1);
\draw (1,1)--(2,0);
\draw (1,1)--(2,-1);
\draw (1,-1)--(2,-1);
\draw (2,1)--(3,1);
\draw (2,0)--(3,0);
\draw (2,0)--(3,-1);
\draw (2,-1)--(3,1);
\draw (2,-1)--(3,-1);
\draw (3,1)--(4,1);
\draw (3,0)--(4,0);
\draw (3,0)--(4,-1);
\draw (3,-1)--(4,-1);
\draw (4,1)--(5,1);
\draw (4,0)--(5,-1);
\draw (4,0)--(5,1);
\draw (4,-1)--(5,-1);
\draw (5, 1)--(6, 0);
\draw (5, -1)--(6, 0);
\draw (0,0.3) node{$\varnothing$};
\draw (1,-1.3) node{$\{1\}$};
\draw (1,1.3) node{$\{4\}$};
\draw (2,-1.3) node{$\{1,4\}$};
\draw (2.2,0.4) node{$\{1,2\}$};
\draw (1.9,1.3) node{$\{1,6\}$};
\draw (3,-1.3) node{$\{1,2,6\}$};
\draw (3.3,0.4) node{$\{1,2,4\}$};
\draw (3,1.3) node{$\{1,4,6\}$};
\draw (4.3,-1.3) node{$\{1,2,4,6\}$};
\draw (4.9,0) node{$\{1,2,3,4\}$};
\draw (4.3,1.3) node{$\{1,4,5,6\}$};
\draw (6.1,-1.1) node{$\{1,2,3,4,6\}$};
\draw (6.1,1.1) node{$\{1,2,4,5,6\}$};
\draw (6.9,0.3) node{$\{1,2,3,4,5,6\}$};

\end{tikzpicture}\\
\hline
\raisebox{1.4cm}{
\begin{tabular}{c} $\operatorname{shift}\bar{c}$\\ \\$(2, 2, 1, 1)$\end{tabular} }
 & \raisebox{.8cm}{
 \begin{tikzpicture}[scale=0.7]
\fill(0,0) circle(.1);
\fill(1,1) circle(.1);
\fill(2,2) circle(.1);
\fill(3,1) circle(.1);
\fill(4,0) circle(.1);
\fill(5,1) circle(.1);
\fill(6,0) circle(.1);
\draw (0,0)--(2,2)--(4,0)--(5,1)--(6,0);
\draw (0,-.5) node{$3$};
\draw (1,0.5) node{$2$};
\draw (2,1.5) node{$1$};
\draw (3,0.5) node{$6$};
\draw (4,-0.5) node{$5$};
\draw (5,0.5) node{$4$};
\draw (6,-0.5) node{$3$};
\fill[pea](-.2,0)--(0,-.2)--(.2,0)--(0,.2)--(-.2,0);
\fill[pea](5.8,0)--(6,-.2)--(6.2,0)--(6,.2)--(5.8,0);
\end{tikzpicture}}&
\begin{tikzpicture}[scale=1.1]
\fill(0,0) circle(.1);
\fill(1,1) circle(.1);
\fill(1,-1) circle(.1);
\fill(2,1) circle(.1);
\fill(2,0) circle(.1);
\fill(2,-1) circle(.1);
\fill(3,1) circle(.1);
\fill(3,0) circle(.1);
\fill(3,-1) circle(.1);
\fill(4,1) circle(.1);
\fill(4,0) circle(.1);
\fill(4,-1) circle(.1);
\fill(5,1) circle(.1);
\fill(5,-1) circle(.1);
\fill(6,0) circle(.1);
\draw (0,0)--(1,1);
\draw (0,0)--(1,-1);
\draw (1,1)--(2,1);
\draw (1,1)--(2,0);
\draw (1,-1)--(2,0);
\draw (1,-1)--(2,-1);
\draw (2,1)--(3,1);
\draw (2,0)--(3,0);
\draw (2,-1)--(3,0);
\draw (2,-1)--(3,-1);
\draw (3,1)--(4,1);
\draw (3,0)--(4,0);
\draw (3,-1)--(4,-1);

\draw (3,-1)--(4,0);
\draw (3,1)--(4,-1);

\draw (4,1)--(5,1);
\draw (4,0)--(5,1);
\draw (4,-1)--(5,-1);
\draw (4,-1)--(5,1);

\draw (5, 1)--(6, 0);
\draw (5, -1)--(6, 0);

\draw (0,0.3) node{$\varnothing$};
\draw (1,-1.3) node{$\{5\}$};
\draw (1,1.3) node{$\{3\}$};
\draw (1.9,-1.3) node{$\{3,5\}$};
\draw (1.9,0.4) node{$\{5,6\}$};
\draw (2,1.3) node{$\{2,3\}$};
\draw (3,-1.3) node{$\{3,4,5\}$};
\draw (3.2,0.4) node{$\{3,5,6\}$};
\draw (3,1.3) node{$\{2,3,5\}$};
\draw (4.3,-1.3) node{$\{2,3,5,6\}$};
\draw (4.9,0) node{$\{3,4,5,6\}$};
\draw (4.3,1.3) node{$\{2,3,4,5\}$};
\draw (6.1,-1.1) node{$\{2,3,4,5,6\}$};
\draw (6.1,1.1) node{$\{1,2,3,5,6\}$};
\draw (6.9,0.3) node{$\{1,2,3,4,5,6\}$};

\end{tikzpicture}\\
\hline
\end{tabular}
\caption{Example showing that lattices of upper and lower ideals can be non-isomorphic (the inclusion in the Hasse diagrams given is in the direction left to right)}
\end{table}

A second, this time bijective proof of Theorem \ref{thm:sym} was given by Elizalde and Sagan in \cite{es22}. Interestingly, both proofs of this result are highly intricate and we felt it natural to seek a transparent proof of this basic result. We present such a proof in this paper, see Corollary \ref{cor:main}.

As mentioned above, in \cite{main},  the symmetry of the rank polynomials of circular fence posets was used to prove Theorem \ref{thm:1}, that rank polynomials of (regular) fence posets are unimodal. Generically, rank polynomials of circular fence posets seemed to be unimodal as well, though there are certain exceptions; a calculation shows that 
\[\overline{R}((1, 1, 1, 1);q) = 1 + 2q + q^2 + 2q^3 + q^4.\]
Extensive computer calculations however suggested the following conjecture. 
\begin{conj}[\cite{main}] \label{conj:unimodality}
The rank polynomial $\overline{R}(\bar{a};q)$ of a circular fence poset $\overline{F}(\bar{a})$ is unimodal except when $\bar{a} = (a, 1, a, 1)$ or $(1, a, 1, a)$ for some positive integer $a$. 
\end{conj}
In this same paper, the authors were able to use the close connection between circular fences posets and regular fence posets to show that if $\bar{R}(\bar{a})$ is not unimodal, then the composition $\bar{a}$ must be of the form 
\begin{align}\label{comp}
\bar{a} = (a_1, \overbrace{1}^{b_1}, a_2, \overbrace{1}^{b_2}, \ldots, a_s, \overbrace{1}^{b_s}),
\end{align}
where any two entries larger than $1$ are seperated by at least one $1$. In other words $b_i$ are at least $1$.

While the authors were able to prove certain additional necessary conditions for non-unimodality, they were unable to settle the conjecture. For the purposes of this paper, we will mainly focus on the case where each $b_i=1$, so that we get a composition of the form:
\begin{align}\label{comp2}
\bar{a} = (a_1, 1, a_2, 1, \ldots, a_s, 1).
\end{align}
It is interesting that such compositions also play an important role in the bijective proof of symmetry of Elizalde and Sagan in \cite{es22}, where the authors refer to such circular fence posets as \definition{gate posets}. 

It turns out that ideals of circular fence posets coming from gate posets are precisely the lattice points in the chainlink polytope $\CL(\bar{a}, 1)$, where $\bar{a} = (a_1, a_2, \ldots, a_s)$. Further, the number of rank $k$ ideals is precisely the number of lattice points in $\CL^k(\bar{a}, 1)$. The symmetry of the rank polynomials of gates may then be written as 
 \[\#\CL^k(\bar{a}, 1) = \#\CL^{n-k}(\bar{a},1),\]
 where $n = a_1 + a_2 + \ldots + a_s + s$ and $\#P$ for a polytope $P$ is the number of lattice points contained in $P$. We will see in Proposition \ref{prop:vertices} that rank the polytopes $\CL^k(\bar{a}, 1)$ while not necessarily integral, are always half integral, in particular, rational. 
 
 We were naturally led to investigate whether the syntactic generalization 
 \begin{align}\label{stmt:Ehr}\operatorname{Ehr}\, \CL^k(\bar{a}, 1) = \operatorname{Ehr}\, \CL^{n-k}(\bar{a},1),
 \end{align}
 is true as well. Well, it is! And this is the content of the main theorem in this paper. The equality of Ehrhart quasipolynomials yields as a corollary the equality of volumes of these polytopes, see Theorem \ref{thm:main}.

Proving Theorem \ref{thm:main} needed several new ideas. Denote the Ehrhart polynomial $\operatorname{Ehr}\, \CL^k(\bar{a}, 1)$ by $f_k$. This polynomial evaluated at $1$ counts the number of ideals of size $k$ in a certain circular fence poset. When evaluated at other integers, say $f_k(m)$, we will show that the value can again be interpreted as the number of lower ideals of size $mk$ in a certain poset, which we call a \definition{chainlink poset}.  These posets share a familial resemblence to circular fence posets : They are introduced in Section \ref{sec:4}, where we also discuss the connections to circular fence posets. An example of a chainlink poset is given below.

\begin{figure}[H]
 \begin{tikzpicture}[scale=.7]
\fill(0,0) circle(.1);
\fill(1,1) circle(.1);
\fill(2,2) circle(.1);
\fill(3,1) circle(.1);
\fill(2,0) circle(.1);

\fill(4,2) circle(.1);
\fill(5,3) circle(.1);
\fill(6,4) circle(.1);
\fill(-1,-1) circle(.1);

\fill(6,2) circle(.1);
\fill(7, 3) circle(.1);
\draw (1,1)--(2,0)--(3,1);

\draw (-1, -1)--(0,0)--(2,2)--(3,1)--(4,2)--(5, 3)--(6, 4)--(7, 3)--(6, 2)--(5, 3);
\draw (-1,-1.5) node{$1$};

\draw (0,-.5) node{$2$};
\draw (1,0.5) node{$3$};
\draw (2,1.5) node{$4$};

\draw (3,0.5) node{$6$};
\draw (2,-0.5) node{$5$};

\draw (4,1.5) node{$7$};
\draw (5,2.5) node{$8$};
\draw (6,3.5) node{$9$};

\draw (6,1.5) node{$1$};
\draw (7,2.5) node{$2$};

\draw[magenta] (0,0) circle (6pt);
\draw[magenta] (0,0) circle (8pt);

\draw[magenta] (7,3) circle (6pt);
\draw[magenta] (7,3) circle (8pt);

\fill[pea](-1.2,-1)--(-1,-1.2)--(-.8,-1)--(-1,-.8)--(-1.2,-1);
\fill[pea](5.8,2)--(6,1.8)--(6.2,2)--(6,2.2)--(5.8,2);

\end{tikzpicture} 
\caption{A chainlink poset: $C((4, 5), 2)$}
\end{figure}
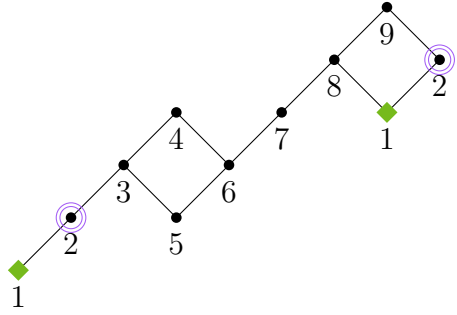

In this figure, cover relations are represented as usual by upward sloping lines : We have for instance that $5 \prec 3$ and $6 \prec 7$. When two vertices have the same label, this means that they are identified. The rank polynomial of the above chainlink poset $C((5, 4), 2)$ is $1 + 2q + 3q^2 + 3q^3 +3q^4 + 3q^5 + 3q^6 + 3q^7 + 2q^8 + q^9$. Note that this polynomial is symmetric. 

We will show that all chainlink posets have symmetric rank polynomials. The strategies for showing symmetry for circular fence posets in \cite{main, es22} do not carry over and we needed to approach the problem differently. The new ingredient is a linear algebraic approach coming from the theory of oriented posets (see  \cite{ops}) that arguably yields a transparent proof. We note that this yields a new (third) proof of rank symmetry for circular fences as well. 

This approach has as its starting point the following basic feature of fence posets: They can be built up by gluing chains in an iterative manner. We review in Section \ref{MatFor} how the rank polynomials of fence (and chainlink) posets can be computed by multiplying certain $2 \times 2$ matrices: The entries of these matrices are certain polynomials that  encode refined order relations. 

Using this approach has led to another felicitous consequence : We discovered new recurrences, that we could then use to prove Conjecture \ref{conj:unimodality}. We include a proof of this in Section \ref{Sec:Unimodality}. 

\section{Chainlink polytopes}\label{sec:3}
We repeat here the description of the chainlink polytopes. 
\begin{defn}
Let $\bar{a} = (a_1, \ldots, a_s)$ be a composition of $n$ and let $l \geq 0$. The \definition{chainlink polytope} $\CL(\bar{a}, l)$ is defined as 
\[\CL(\bar{a},l)= \{x \in \mathbb{R}^s \mid 0 \leq x_i \leq a_i, \, i \in [s], \, x_{i} - x_{i \,(\operatorname{mod} s)+1} \leq a_{i}-l, \, i \in [s]\}.\]
Given $t \in \mathbb{R}_{\geq 0}$, the $t$-section of the chainlink polytope is the polytope defined as 
\[\CL^t(\bar{a},l)=\CL(\bar{a},l) \cap \{x_1 + \ldots + x_s = t\}.\]
\end{defn}
We note here a basic integrality property of these polytopes. 
\begin{prop}\label{prop:vertices}
 Let $\bar{a} = (a_1, \ldots, a_s)$ be a composition of $n$, let $l, t \in \mathbb{Z}_{\geq 0}$. Then 
 \begin{itemize}
 \item The vertices of $\CL(\bar{a}, l)$ are integral.
\item Assume that $2l \leq \operatorname{min} \bar{a}$. Then for every integer $t$, the vertices of $\CL^t(\bar{a}, l)$ are either integral or half integral. 
 \end{itemize}
\end{prop}
We prove this following \ref{vertexlemma}. 
Some remarks are in order. 
\begin{enumerate}
\item Half-integral vertices can indeed occur. For instance, the polytope $\CL^2((2, 2), 1)$, is given by 
\[0 \leq x_1 \leq 2, \quad 0 \leq x_2 \leq 2, \quad x_1 - x_2 \leq  1, \quad x_1 + x_2 = 2,\]
 has dimension $1$ and has as its' two vertices $(3/2, 1/2)$ and $(1/2, 3/2)$.
\item The condition on $l$ is necessary. Take for instance the polytope $\CL^5((3, 3, 3), 2)$. One may check that the point $(8/3, 5/3, 2/3)$ is a vertex of the polytope. 
\item It is possible to show that if the vertices of $\CL^t(\bar{a}, l)$ are all integral or half integral for every integer $l$, then we must have that $2l \leq \operatorname{min} \{a_i\}_{i \in [s]}$. 
\end{enumerate}

These polytopes are closely related to a class of posets we call \emph{Chainlink Posets} which we describe in Section~\ref{sec:4}.

\subsection{ Connections to other polytopes [Ezgi, Mohan]}
There are two well known polytopes associated to any poset, namely the order and chain polytopes. Recall that given a poset $\mathcal{P}$ on $[n]$, the order polytope is defined as 
\[\mathcal{O}(\mathcal{P}) := \{ x \in [0, 1]^n \mid \forall i \prec_{\mathcal{P}} j, \, \, x_i \leq x_j\}.\]
Recall that a chain in a poset is a totally ordered subset; let us denote the set of all chains in the poset $\mathcal{P}$ by $\operatorname{Ch}(\mathcal{P})$. The Chain polytope is defined as
\[\mathcal{C}(\mathcal{P}) := \{ x \in [0, 1]^n \mid \forall C \in \operatorname{Ch}(\mathcal{P}), \, \, \sum_{i \in C}  x_j \leq 1\}.\]
These polytopes are apriori quite distinct from each other. The vertices of the order polytope are easily seen to be in bijection with the set of upper ideals (upper closed subsets) of the poset $\mathcal{P}$; each upper ideal gives the vertex $a_I$, where 
\[a_{I}(i) = \begin{cases} 
1 & i \in I,\\
0 & \text{otherwise}
\end{cases}\]
The vertices of the chain polytope on the other hand are indexed by the anti-chains (subsets where any two distinct elements are incomparable). Each antichain $J$ yields the vertex $b_J$, where
\[b_{J}(j) = \begin{cases} 
1 & j \in J,\\
0 & \text{otherwise}
\end{cases}\]
Despite all this, one has the following delightful theorem of Stanley \cite{stanley1986two}. 
\begin{thm}[Stanley]
The order and chain polytopes of any poset have the same volume. 
\end{thm}
It is easy to get a combinatorial expression for the volume of the order polytopes. Any total order $L$ on $\mathcal{P}$ yields a simplex $\mathcal{S}(L)$ 
\[\mathcal{S}(L) =  \{ x \in [0, 1]^n \mid \forall i \prec_{L} j, \, \, x_i \leq x_j\},\]
and all these simplices are relatively disjoint. Consequently, if we let $\mathcal{L}(\mathcal{P})$ be the set of all linear extensions of $\mathcal{P}$, we have that
\[|\mathcal{O}(\mathcal{P})| = \dfrac{\# \mathcal{L}(\mathcal{P})}{n!}.\]
Stanley's proof exhibits a unimodular triangulation of the chain polytope as well, indexed by the linear extensions of $\mathcal{P}$.

\subsection{Polyhedral models for fence posets}
The first attempt at getting a polyhedral perspective on fence posets would be to consider the order polytope of the fence poset (call the fence poset $\mathcal{F}$ and the order polytope $\mathcal{O}(\mathcal{F})$ and note that the number of rank $k$ ideals is the number of lattice points in 
\[\mathcal{O}^k(\mathcal{F}) = \mathcal{O}(\mathcal{F}) \cap \{x_1 + \ldots + x_s = k\}.\]
However, though we have that $\mathcal{O}^k(\mathcal{F})$ and $\mathcal{O}^{n-k}(\mathcal{F})$ have the same number of lattice points, it is \emph{not true} that they have the same volume. 

 \begin{rmk}\label{rem:order}
 Consider the gate poset for the composition $\bar{c} = (3,1,2,1,1,1)$, let $P = \mathcal{O}(\bar{\mathcal{F}}(\bar{c}))$ be the order polytope and let $P^k = P \cap \{x_1 + \ldots + x_n = k\}$. Then the  lattice points in the second dilates of $P^{4}$ and the complementary section $P^{5}$ number respectively $84$ and $83$, i.e. we do not have equality of Ehrhart polynomials. This indicates the subtlety involved in this problem. 
 \end{rmk}

 So far, we have proposed a polytopal model for \emph{gates}, which are circular fences coming from compositions of the form $(c_1, 1, \ldots, c_s, 1)$, i.e. where all the down steps have size $1$. As noted in the paper in Remark \ref{rem:order}, the order polytope for the gate poset \emph{does not} obey symmetry, which is why we worked with chainlink polytopes instead. 
 
What about general circular fences, those coming from compositions of the form $(c_1, d_1, \ldots, c_s, d_s)$ with differing lengths of down steps $d_1, \ldots, d_s$? A natural proposal is as follows. The polytope  will consist of all real tuples $(x_1, y_1, \ldots, x_s, y_s)$ such that 
\[0 \leq x_i \leq a_i + 1, \quad 0 \leq y_i \leq b_i - 1, \quad (b_i-1)(x_i - a_i) \leq y_i , \quad y_i \leq (b_i- 1)x_{i+1}.\]
where the indices are taken cyclically. 

When all the down steps are $1$, all save the first set of inequalities become trivial, but if agree to combine the last two set of inequalities, we get $x_i - a_i \leq x_{i+1}$, which are the defining equations in the chainlink polytope with link number $1$. 

Unfortunately, these polytopes do not have the symmetry that chainlink polytopes have. We do not know if there is a way of defining polyhedral models for general circular fence posets so that this symmetry does hold.

\section{chainlink Posets and Rank Symmetry}\label{sec:4}

Let $\HL(\bar{a},l)$ be a chainlink polytope with  $2l\leq \mathrm{min}_i(a_i)$. Consider the integer points that lie inside the polytope.  When $l=1$, these points correspond to ideals of the circular fence poset $F(a_1-1,1,a_2-1,1,\ldots,a_s-1,1)$, where the rank of the ideal corresponds to the sum of the coordinates of the point. For general $l$, the integer points can be interpreted as ideals of a poset $F_l(\bar{a})$ formed by adding extra edges to the Hasse diagram of $F(a_1-1,1,a_2-1,1,\ldots,a_s-1,1)$ as shown in Figure~\ref{fig:posetexample}. More precisely, we can define chainlink posets as follows:
\begin{defn} Let $\bar{a}$  be a composition, and $l$ be a positive integer satisfying $2l\leq \operatorname{min}_i{a_i}$ as in the case for chainlink polytopes. The \definition{chainlink poset} $\CLP(\bar{a},l)$ is given by points $x_{i,j}$ for $1\leq i \leq \ell(\bar{a})$ and $0\leq j\leq a_i$  with the generating relations $x_{i,0}\succeq x_{i,1} \succeq \cdots \succeq x_{i,a_i}$ and $x_{i,a_i-l}\geq x_{i+1,l}$ for each $i$ where $i+1$ is calculated cyclically.
\end{defn}  

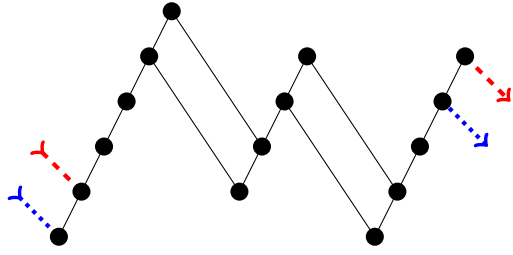
\begin{figure}
    \centering
    \begin{tikzpicture}[scale=.6]
		\node  (1) at (-16.5, 1) {};
		\node  (2) at (-16, 2) {};
		\node  (3) at (-17, 0) {};
		\node  (4) at (-15.5, 3) {};
		\node  (5) at (-15, 4) {};
		\node  (6) at (-14.5, 5) {};
		\node  (7) at (-13, 1) {};
		\node  (8) at (-12.5, 2) {};
		\node  (9) at (-12, 3) {};
		\node  (10) at (-11.5, 4) {};
		\node  (11) at (-10, 0) {};
		\node  (12) at (-9.5, 1) {};
		\node  (13) at (-9, 2) {};
		\node  (14) at (-8.5, 3) {};
		\node  (15) at (-8, 4) {};
		 \draw[->,red,dashed,ultra thick] (15.center) to (-7,3);
		\draw[>-,red,dashed,ultra thick] (-17.5,2)--(-16.5,1);
		\draw[->,blue,dotted, ultra thick] (14.center) to (-7.5,2);
		\draw[>-,blue,dotted,ultra thick] (-18,1)--(-17,0);
		\fill(-16.5,1) circle(.2); 
	    \fill(-16,2) circle(.2); 
		\fill(-17,0) circle(.2); 
		\fill(-15.5,3) circle(.2); 
	    \fill(-15,4) circle(.2); 
		\fill(-14.5,5) circle(.2); 
		\fill(-13,1) circle(.2); 
	    \fill(-12.5,2) circle(.2); 
		\fill(-12,3) circle(.2); 
		\fill(-11.5,4) circle(.2); 
	    \fill(-10,0) circle(.2); 
		\fill(-8,4) circle(.2); 
		\fill(-9.5,1) circle(.2); 
	    \fill(-9,2) circle(.2); 
		\fill(-8.5,3) circle(.2); 
		\draw (3.center) to (6.center);
		\draw (6.center) to (8.center);
		\draw (5.center) to (7.center);
		\draw (7.center) to (10.center);
		\draw (10.center) to (12.center);
		\draw (9.center) to (11.center);
		\draw (11.center) to (15.center);
    \end{tikzpicture}
    \caption{The chainlink poset with $\bar{a}=(6,4,5)$ and $l=2$, the two top right nodes are connected to the two bottom left nodes.}
    \label{fig:posetexample}
\end{figure}

We will use $\HL(a,l)^t$ to denote the slice of the polytope with respect to the hyperplane $x_1 + \ldots + x_s = t$. Note that this slice can be non-empty only when $t\in[0,n]$. Furthermore, the number of integer points in $\HL^t(a,l)$ is given by the coefficient of $q^t$ in the rank polynomial of $\CLP(a,l)$. 

 The fact that 
\[\# \HL^t(\bar{a},1) = \# \HL^{n-t}(\bar{a},1)\]
where  $\# P$ for a polytope $P$ is the number of lattice points in $P$ is as a result equivalent to the symmetry of the rank polynomial of  $\CLP(a,1)$ which was recently proved using an inductive argument by Kantarcı Oğuz and Ravichandran and then bijectively by Elizalde and Sagan. 

The connection between integer points of the polytope and rank polynomial of the corresponding poset still holds if we multiply all by a number $k$. This allows us to describe the coefficients of the Ehrhart quasi polynomial of slices of the chainlink polytope in terms of coefficients of rank polynomials of some chainlink posets. That means a general statement about the symmetry of rank polynomials of all chainlink posets can be used to prove the main theorem stated in the introduction, which is precisely what we aim to do in the next few sections.

\subsection{Matrix Formulation}\label{MatFor}

An oriented poset $ {P\!\nearrow}=(P,x_L,x_R)$ consists of a poset $P$ with two specialized vertices $x_L$ and $x_R$ which can be thought as the target (left) vertex $\target\,$  and the source (right) vertex ${\color{blue}\rightarrow}$. One can think of an oriented poset as a poset with an upwards arrow coming out of the source vertex $x_R$. One can combine oriented matrices by linking the arrow of one poset with target of another via $x_R \preceq y_L$ ($x_R\nearrow y_L$) to get $ ({P}\!\nearrow {Q})\!\nearrow$\footnote{Linking via  $x_R \succeq y_L$ is also an option, see \cite{ops}.  }. 

The effect of this operation on the rank polynomial can be calculated easily by $2\times2$ matrices. A \definition{rank matrix} of an oriented poset ${P\!\nearrow}$ is defined as follows:

    $$	\displaystyle \drm( {P}\!\nearrow):=\begin{bmatrix} \rank( {P};w)|_{x_R\in I} & \rank( {P};w)|_{x_R\notin I}\\ \rank( {P};w)|_{\substack{x_R\in I\\x_L \notin I}} & \rank( {P};w)|_{\substack{x_R\notin I\\x_L \notin I}}
			\end{bmatrix}$$

The entries are partial rank polynomials, where we are restricting to the ideals of the poset $P$ satisfying the given constraints. We also use the notation $\close( {P}\searrow)$ (resp. $\close( {P}\nearrow)$ to denote the structure obtained by adding the relation $x_R \succeq x_L$ (resp. $x_R \preceq x_L$) . On the rank matrix level, this corresponds to taking the trace. See Table~\ref{table:tableofmoves} for precise formulas and examples of these operations.

\begin{table}
    \centering
    \begin{tabular}[t]{|c|c|c|}
\hline
         & Formula & Example \\
         \hline &&\\

    $ {P}\nearrow {Q}$
        &\begin{tabular}{c}

 $\drm(( {P}\nearrow {Q})\!\nearrow)=\drm( {P}\!\nearrow) \cdot \drm( {Q}\!\nearrow)$
 \end{tabular}&\begin{tabular}{c}
 \begin{tikzpicture}[scale=.45]
\draw (0,0)--(1,1)--(3,-1);
\fill[white] (0,0) circle(.2) ;
\fill[red] (0,0) circle(.1) ;
\draw[red] (0,0) circle(.2);
\fill (1,1) circle(.1)  ;
\fill (2,0) circle(.1)  ;
\fill[blue] (3,-1) circle(.15)  ;
\draw[->, blue,dotted, thick] (3,-1)--(3.5,-.5);
\node at (1.5,-1.5) {$\scriptstyle{ {P}\!\nearrow}$\quad,};
\end{tikzpicture}\quad \begin{tikzpicture}[scale=.45]
\fill[white] (0,1) circle(.2) ;
\draw (0,-.5)--(1,.5);
\fill[white] (0,-.5) circle(.2) ;
\fill[red] (0,-.5) circle(.1) ;
\draw[red] (0,-.5) circle(.2);
\fill[blue] (1,.5) circle(.15)  ;
\draw[->, blue,dotted,thick] (1,.5)--(1.5,1);
\node at (.5,-1.5) {$\scriptstyle{ {Q}\!\nearrow}$};
\end{tikzpicture}$\quad$ \raisebox{.2cm}{$\rightarrow$} $\quad$\begin{tikzpicture}[scale=.45]
\draw (0,0)--(1,1)--(3,-1)--(5,1);
\fill[white] (0,0) circle(.2) ;
\fill[red] (0,0) circle(.1) ;
\draw[red] (0,0) circle(.2);
\fill (1,1) circle(.1)  ;
\fill(2,0) circle(.1);
\fill(3,-1) circle(.1);
\fill (4,0) circle(.1) ;
\fill[blue] (5,1) circle(.15)  ;
\draw[->, blue,dotted, thick] (5,1)--(5.5,1.5);
\node at (3,-1.5) {$\scriptstyle{( {P}\nearrow {Q})\!\nearrow}$};
\end{tikzpicture} \end{tabular}\\
\hline &&\\
 $\close\,( {P}\nearrow)$ &$\rank( \close( {P}\nearrow);w)=\operatorname{tr}(\rmm({ {P}\!\nearrow}))$ &\begin{tabular}{c}
        \begin{tikzpicture}[scale=.45]
\draw (0,0)--(1,1)--(3,-1);
\fill[white] (0,0) circle(.2) ;
\fill[red] (0,0) circle(.1) ;
\draw[red] (0,0) circle(.2);
\fill (1,1) circle(.1)  ;
\fill (2,0) circle(.1)  ;
\fill[blue] (3,-1) circle(.15)  ;
\draw[->, blue,dotted, thick] (3,-1)--(3.5,-.5);
\node at (1.5,-1.5) {$\scriptstyle{ {P}\!\nearrow}$};
\end{tikzpicture} $\quad$ \raisebox{.2cm}{$\rightarrow$} $\quad$   \begin{tikzpicture}[scale=.45]
\draw (0,0)--(1,1)--(3,-1)--(0,0);
\fill (0,0) circle(.1) ;
\fill (1,1) circle(.1)  ;
\fill (2,0) circle(.1)  ;
\fill (3,-1) circle(.1);
\node at (1.5,-1.5) {$\scriptstyle{\close( {P}\!\nearrow)}$};
\end{tikzpicture}\end{tabular} \\
\hline
    \end{tabular}
    \caption{The moves are shown through examples}
    \label{table:tableofmoves}
\end{table}

In particular, consider the case where $P$ is formed of a single node equal to both $x_R$ and $x_L$. We call this oriented poset an \definition{up step} and denote the corresponding matrix by $U$.

    $$	\displaystyle U:=\drm(\bullet\!\nearrow):=\begin{bmatrix}q&1\\0&1
			\end{bmatrix}.$$

Note that combining $k+1$ such posets gives us a chain of length $k$ with $x_L$ corresponding to the minimal element, and $x_L$ to the maximal. In the matrix level we have:

$$ \drm(C_k\!\nearrow)=U^{k-1}.$$

Let $ {B_{a\times b}}\!\nearrow$ denote the $ab$-element oriented box poset given by the direct product of two chains $C_{a-1}$ and $C_{b-1}$ with the left vertex given by $(a-1,0)$ and the right vertex is given by $(0,b-1)$ ($B_{3\times4}\!\nearrow$ is shown in Figure~\ref{fig:boxexample}). 

\begin{figure}
    \centering
    \rotatebox{-30}{
    \begin{tikzpicture}[scale=.6]
		\node  (2) at (-16, 2) {};
		\node  (4) at (-15.5, 3) {};
		\node  (5) at (-15, 4) {};
		\node  (6) at (-14.5, 5) {};
		\node  (7) at (-13, 1) {};
		\node  (8) at (-12.5, 2) {};
		\node  (9) at (-12, 3) {};
		\node  (10) at (-11.5, 4) {};
		\draw (2.center) to (6.center);
		\draw (4.center) to (8.center);
		\draw (2.center) to (7.center);
		\draw (7.center) to (10.center);
		\draw (10.center) to (6.center);
		\draw (9.center) to (5.center);
		\draw (-13,4.5) to (-14.5,1.5);
		\fill(-15.5,3) circle(.15); 
	    \fill(-15,4) circle(.15); 
		\fill(-14.5,5) circle(.15); 
		\fill(-13,1) circle(.15); 
	    \fill(-12.5,2) circle(.15); 
		\fill(-12,3) circle(.15); 
		\fill[blue](-11.5,4) circle(.22); 
	    \fill(-14.5,1.5) circle(.15); 
		\fill(-14,2.5) circle(.15); 
		\fill(-13.5,3.5) circle(.15); 
	    \fill(-13,4.5) circle(.15); 
	\fill[white] (-16,2) circle(.3) ;
\fill[red] (-16,2) circle(.15) ;
\draw[red] (-16,2) circle(.3);
\draw[->, blue, thick] (-11.5, 4)--(-11, 5);
    \end{tikzpicture}}
    \caption{The box poset $B_{3\times 4}\!\nearrow$.}
    \label{fig:boxexample}
\end{figure}
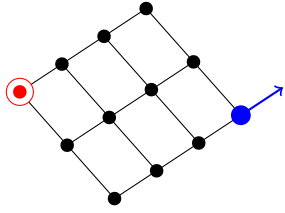

The rank matrix of a box poset is given as follows:
    $$	\displaystyle \drm(B_{a\times b}\!\nearrow)=\begin{bmatrix} q^b\qbinom{a+b-1}{b}_q & \qbinom{a+b-1}{b-1}_q  \\  q^b\qbinom{a+b-2}{b}_q  & \qbinom{a+b-2}{b-1}_q \end{bmatrix} = \begin{bmatrix} \qbinom{a+b}{b}_q -\qbinom{a+b-1}{b-1}_q & \qbinom{a+b}{b}_q- q^b\qbinom{a+b-1}{b}_q  \\  q^b\qbinom{a+b-2}{b}_q  & \qbinom{a+b-2}{b-1}_q 
			\end{bmatrix}$$

Next we will see that any given chainlink poset can be realized by combining copies of the up step, and the box poset $B_{2\times l}\!\nearrow$.

\begin{prop}  Consider the chainlink poset $\CLP(\bar{a},l)$ with $2l\leq \operatorname{min}_i(a_i)$. Let $\bar{d}$ be the weak composition formed by taking $d_i=a_i-2l$.  The rank polynomial of $\CLP(\bar{a},l)$ is given by:

\begin{equation} \label{eq:matrixconnection}
    \rank (\CLP(\bar{a},l);q)= \operatorname{tr}(U^{d_1}\cdot B \cdot U^{d_2} \cdot B \cdot \,\cdots\,\cdot U^{d_1} \cdot B)
\end{equation} where $B$ denotes the rank matrix $\drm(B_{2\times l}\!\nearrow)$.

\end{prop}

\begin{proof} The description of $\CLP(\bar{a},l);q$ is given by taking chains and connecting the $l$ maxima of a chain of the $l$ minima of the next. The condition on $l$ assures that these are always different vertices. Each pair of $2l$ vertices has the shape of the poset $(B_{2\times l}\!\nearrow$ where the right endpoint of $i$th one is connected to the  left endpoint of the $i+1$st one from below, with $d_i$ points in between. This is the poset given by $\close(B_{2\times l}\!\nearrow  C_{d_1-1}\!\nearrow B_{2\times l}\!\nearrow U C_1-{d_2-1}\!\nearrow\cdots B_{2\times l}\!\nearrow U C_{d_s-1}\!\nearrow)$. As the rank polynomial can be calculating by taking the trace of the corresponding matrices, and $\rmm(C_{a-1}\!\nearrow)=U^a$ the result follows.

\end{proof}

\begin{example} The chainlink poset given in Figure~\ref{fig:posetexample} with $\bar{a}=(6,4,5)$ and $l=2$ can be formed by combining $2\times 2$ boxes with up steps and then taking the closure: $\close(B_{2\times 2}\!\nearrow \cdot \bullet\!\nearrow \cdot  \bullet\!\nearrow \cdot B_{2\times 2}\! \nearrow \cdot B_{2\times l}\!\nearrow \cdot  \bullet\!\nearrow)$. The corresponding rank polynomial is given by:

\begin{eqnarray*}
  &&\operatorname{tr} \left( \begin{bmatrix} q^2 [3]_q & [3]_q\\ q^2 & [2]_q\end{bmatrix} \cdot \begin{bmatrix} q&1\\0&1 \end{bmatrix} \cdot \begin{bmatrix} q&1\\0&1 \end{bmatrix} \cdot \begin{bmatrix} q^2 [3]_q & [3]_q\\ q^2 & [2]_q\end{bmatrix} \cdot \begin{bmatrix} q^2 [3]_q & [3]_q\\ q^2 & [2]_q\end{bmatrix} \cdot \begin{bmatrix} q&1\\0&1 \end{bmatrix} \right)\\
  && =1+3q+6q^2+9q^3+12q^4+14q^5+16q^6+17q^7\\
  && \hspace{4cm} +17q^8+16q^9+14q^{10}+12q^{11}+9q^{12} +6q^{13}+3q^{14}+q^{15}.
\end{eqnarray*}

\end{example}
Note that the rank polynomial given in this instance is symmetric. Next, we will show that this is always the case.

\subsection{Recurrence Relations and Rank Symmetry}

One advantage of building posets via matices is that the characteristic equations of matrices give us recurrence relations in the rank polynomial level. For example, consider the characteristic polynomial of $U$. Plugging $U$ in the place of $x$ gives us the following identity:
\begin{equation*}
 U^2 =(q+1) U + q,
\end{equation*}
 Note that the coefficient of $U$ is symmetric around $q^{1/2}$ and $q$ is trivially symmetric around $q$.

\begin{lemma}\label{lem:characteristic}
Let $B= \drm(B_{a\times b}\!\nearrow)$ for some fixed $a,b$. The characteristic polynomials of $B$ as well as well as $BU$ have coefficients that are symmetric polynomials in $q$. In particular, the trace and determinant of $B$ and $BU$ are symmetric about $ab/2$, $ab$, $(ab+1)/2$ and $ab+1$ respectively.
\end{lemma}

\begin{proof} As the matrix $B$ satisfies 
\begin{eqnarray*}\displaystyle \drm(B_{a\times b}\!\nearrow)=\begin{bmatrix} q^b\qbinom{a+b-1}{b}_q & \qbinom{a+b-1}{b-1}_q  \\  q^b\qbinom{a+b-2}{b}_q  & \qbinom{a+b-2}{b-1}_q \end{bmatrix} = \begin{bmatrix} \qbinom{a+b}{b}_q -\qbinom{a+b-1}{b-1}_q & \qbinom{a+b-1}{b-1}_q  \\ \qbinom{a+b-1}{b}_q-\qbinom{a+b-2}{b-1}_q  & \qbinom{a+b-2}{b-1}_q 
			\end{bmatrix},
\end{eqnarray*}

we will define another matrix $B'$ to simplify some of our calculations:
\begin{eqnarray*}
 \displaystyle B'&:=& B \begin{bmatrix} 1&-1\\1&0\end{bmatrix} =     \begin{bmatrix}
    \qbinom{a+b}{b}_q &-q^b\qbinom{a+b-1}{b}_q\\
    \qbinom{a+b-1}{b}_q &-q^b\qbinom{a+b-2}{b}_q
    \end{bmatrix}.   
\end{eqnarray*}

Consider the trace of $B$.
\begin{eqnarray*} {tr}(B) &=& q^b \qbinom{a+b-1}{b}_q + \qbinom{a+b-2}{b-1}_q\\
&=&\qbinom{a+b}{b}_q- \qbinom{a+b-1}{b-1}_q +\qbinom{a+b-1}{b-1}_q - q^a\qbinom{a+b-2}{a}_q\\
&=&\qbinom{a+b}{b} - q^a\qbinom{a+b-2}{a}_q.
\end{eqnarray*}
As both polynomals on the right are symmetric about $ab/2$ the result follows. 
Similarly, we can show that the determinant gives us a symmetric polynoial around about $ab$:
\[\operatorname{det}(B) = -\operatorname{det}(B')=q^b\qbinom{a+b}{b}_q   \qbinom{a+b-2}{b}_q - q^b\left(\qbinom{a+b-1}{b}_q  \right)^2. \]

Now let us do the same verification for the trace and determinant of the matrix $BU$:
$$ BU= \begin{bmatrix}  q^{b+1}\qbinom{a+b-1}{b}_q & q^b\qbinom{a+b-1}{b}_q + \qbinom{a+b-1}{b-1}_q \\  q^{b+1} \qbinom{a+b-2}{b}_q &  q^b\qbinom{a+b-2}{b}_q  + \qbinom{a+b-2}{b-1}_q\end{bmatrix}=\begin{bmatrix}  q^{b+1}\qbinom{a+b-1}{b}_q & \qbinom{a+b}{b}_q \\  q^{b+1} \qbinom{a+b-2}{b}_q &  \qbinom{a+b-1}{b}_q \end{bmatrix}.
$$

One can verify that the trace of this matrix is symmetric around $(ba+1)/2$, whereas the determinant is symmetric around $(ba+1)$:
\begin{eqnarray*}
\operatorname{tr}(BU)&=&(q^{b+1}+1) \qbinom{a+b-1}{b}_q,\\
\operatorname{det}(BU)&=&q^{b+1} \left(\qbinom{a+b-1}{b}_q\right)^2- q^{b+1} \qbinom{a+b-2}{b}_q\qbinom{a+b}{b}_q.
\end{eqnarray*}
\end{proof}

\begin{lemma} \label{lem:charpol} Let $X$ be a matrix whose trace is given by a symmetric polynomial on $q$ with center $C$.
Then we have the following:
\begin{enumerate}
    \item If $\mathrm{tr}(BX)$ is symmetric around $C+(ab)/2$, then for all $k$, $\mathrm{tr}(B^k X)$ is symmetric with center $C+k(ab)/2$.
    \item If $\mathrm{tr}(U X)$ is symmetric around $C+1/2$, then for all $k$, $\mathrm{tr}(U ^kX)$ is symmetric with center $C+k/2$.
    \item If $\mathrm{tr}(B U X)$ is symmetric around $C+(ab+1)/2$, then for all $k$, $\mathrm{tr}((B U)^kX)$ is symmetric with center $C+k(ab+1)/2$.
\end{enumerate}
\end{lemma}

\begin{proof} For the first claim, consider the characteristic polynomial of $B$.
$${\displaystyle B^{2}=\operatorname {tr} (B) B-\det(B).}$$
Substitution into the trace equation gives the following identity for any $k\geq 2$:
$$\mathrm{tr}(B^kX)=\mathrm{tr} (B)\mathrm{tr}( B^{k-1}X)
-\det(B)\mathrm{tr}(B^{k-2}X)$$
We have shown in Lemma~\ref{lem:characteristic} $\mathrm{tr}(B)$ is a symmetric polynomial with center $ab/2$ and  $\mathrm{det}(B)$ is symmetric with center $ab$, the result follows by induction. The other claims follow similarly as $\mathrm{tr}(U)$, $\det(U)$ and $\mathrm{tr}(BU)$ and $\mathrm{det}(BU)$ are symmetric polynomials with centers of symmetry given by $1/2$, $1$, $(ab+1)/2$ and $ab+1$ respectively.
\end{proof}

\begin{thm}\label{prop:symmetry} Let $(d_1,d_2,\ldots,d_s)$ be a weak composition and 
$B_{a\times b}\!\nearrow$ be an oriented box poset where $B$ denotes the rank matrix $\drm(B_{a\times b}\!\nearrow)$. Then the following polynomial is symmetric.
\begin{equation*}
    \operatorname{tr}(U^{d_1} \cdot B \cdot U^{d_2} \cdot B \cdots U^{d_1} \cdot B).
\end{equation*}
\end{thm}
\begin{proof} We can use Lemma~\ref{lem:charpol} to simplify the statement of our theorem:
\begin{itemize} \setlength\itemsep{0em}
    \item By Lemma~\ref{lem:charpol} (2), it is enough to consider the cases where each $d_i=0$ or $1$.
    \item If however some $d_i$=0, we get consecutive copies of $B$. By  Lemma~\ref{lem:charpol} (1), these cases can be simplified, so that we can assume all $d_i$ are equal to $1$, leaving us with $\operatorname{tr}( (U\cdot B)^k)$ for some $k$.
    \item By Lemma~\ref{lem:charpol} (3), this can be further reduced to the symmetry of $\operatorname{tr}(U\cdot B)$ and $\operatorname{tr}(I)$.
\end{itemize}
As we have already shown that the trace of $B\cdot U$ is symmetric in Lemma~\ref{lem:characteristic} and the trace of the identity matrix is just a constant, we are done.
\end{proof}

Note that when $l=1$, we recover the rank symmetry of gate posets. 

\begin{corollary} \label{cor:main} The rank polynomial of any chainlink poset is symmetric.
\end{corollary}

Not we can use this machinery to prove our main theorem.
\begin{thm}\label{thm:main} Let $\bar{a}$ be a composition of $n$, let $l$ be a positive integer such that $2l \leq \operatorname{min} \{a_i\}_{i \in [s]}$ and let $t$ be a positive real number. Then complementary sections of the chainlink polytope have the same volume,
 \[|\HL^t(\bar{a},l)| = |\HL^{n-t}(\bar{a},l)|,\]
 where $|P|$ for a polytope denotes the relative volume. In fact, their quasi-Ehrhart poynomials are identical.
\end{thm}

\begin{proof} Consider the polytope $k\HL(\bar{a},l)=\HL(k\bar{a},kl)$. The number of integer points in $\HL^t(k\bar{a},kl)$ are given by the coefficient of $q^t$ in the rank polynomial of the corresponding chainlink poset $\CLP(k\bar{a},kl)$. As by Corollary~\ref{cor:main}, the rank polynomial of any chainlink poset is symmetric, the number of integer points in $\HL^t(k\bar{a},kl)$ is the same as the number of integer points in $\HL^{n-t}(k\bar{a},kl)$ for any $k$. As a consequence, they have the same quasi-Ehrhart polynomial.
\end{proof}

\section{Unimodality and Multimodality }\label{Sec:Unimodality}

\subsection{Unimodality}

The recurrence relations from characteristic matrices have other applications as well. In this subsection, we prove the following result.

\begin{thm}\label{unimodality} Rank polynomials of circular fence posets $\bar{F}(\bar{a})$ are unimodal except when $\bar{a} = (a, 1, a, 1)$ or $(1, a, 1, a)$ for some positive integer $a$. 
\end{thm}

We define the matrix for a \definition{down step} denoted by $D$ as follows:

    $$	\displaystyle D:=\begin{bmatrix}1+q&-q\\1&0
			\end{bmatrix}.$$

The following lemma is an easy consequence of the work in \cite{ops}. The interested reader is referred there to learn about how down steps fit into the framework of oriented posets.

\begin{lemma} Let $DC_n$ denote a decreasing chain, an $n$-element chain poset oriented by taking the maximal vertex as the target and the minimal vertex as the source. Then we have,
$$ \drm(DC_n\!\nearrow)=D^{n-1}\cdot U.$$
\end{lemma}

That means the above theorem may be restated as follows:
\begin{reptheorem}{unimodality}
For any composition $\bar{a}$ with an even number of parts the following polynomial is unimodal except when $\bar{a} = (a, 1, a, 1)$ or $(1, a, 1, a)$ for some positive inteİger $a$:
$$\overline{R}(\bar{a},q)=\operatorname{tr}(D^{a_1}U^{a_2}D^{a_3}U^{a_4}\cdots D^{a_{s-1}}U^{a_s})=\operatorname{tr}(U^{a_1}D^{a_2}U^{a_3}D^{a_4}\cdots U^{a_{s-1}}D^{a_s}).$$
\end{reptheorem}

In \cite{main} where this was first stated as a conjecture, it was shown that if the size of the composition is odd the resulting polynomial is indeed unimodal. It was further shown that unimodality holds for even size compositions containing two adjacent parts larger than $1$ or three  adjacant parts $a,1,b$ with $|a-b|>1$. One can also easily show unimodality by direct calculation for the cases of $\bar{a}=(a,b)$ and $\bar{a}=(a,1,a+b,1)$ with $a,b\geq 1$. In this section, we will settle the outstanding cases using recurrence identities similar to those described in the previous section. Our first identity will be the following.

\begin{equation} \tag{Id 1}
    DUD = DU + UD - U + D^3 - D^2 .
\end{equation}

On the rank polynomial level, (Id 1) translates to the following:

\begin{align*} \tag{Id 1$'$}
\overline{R}((a, 1, b, X);q)=& \overline{R}((a-1, 1, b, X);q) + \overline{R}((a, 1, b-1, X);q) \\ -&  \overline{R}((a-1, 1, b-1, X);q) +  \overline{R}((a+b+1,X);q) \\-& \overline{R}((a+b, X);q).
\end{align*}

Here $X$ can be any odd-length composition Note that we allow $a=1$ or $b=1$, in which case, with the assumption that a zero part means the parts to the left and right combine.

\begin{prop}\label{prop:simplify}
For an odd-length sequence $X=(x_1,x_2,\ldots,x_k)$ of positive integers suppose that  $a,b \geq 1$ and $\bar{R}(a-1, 1, b-1, X)$ is unimodal. If $a>1$ or $\ell(X)>1$ with $x_1>1$ pr  $b\geq x_2$  then $\bar{R}(a, 1, b, X)$ is also unimodal.
\end{prop}

\begin{proof}
As unimodality holds for compositions of odd size, we can focus on the case where $2n = |(a, 1, b, X)|$ for some $n$. We can further assume that $|a-b|$ is at most $1$, as otherwise again, we know we have unimodality.

Let $\overline{R}(a-1, 1, b-1, X)$ be unimodal. Then it has a peak at $n-1$. Thus, 
\[[q^{n}]\left( -\overline{R}(a-1, 1, b-1, X) \right)\leq [q^{n-1}] \left(-\overline{R}(a-1, 1, b-1, X)\right).\]
We also have the following by the symmetry of rank polynomials:
\[[q^{n}] \left( \overline{R}(a-1, 1, b, X) + \overline{R}(a, 1, b-1, X)\right) = [q^{n-1}]\left(  \overline{R}(a-1, 1, b, X) + \overline{R}(a, 1, b-1, X)\right),\]

By (Id 1$'$), all that is left to show is that:
\[[q^{n}]\left( \overline{R}(a+b+1, X) - \overline{R}(a+b, X)\right) \geq [q^{n-1}]\left( \overline{R}(a+b+1, X) - \overline{R}(a+b, X)\right).\]
By the symmetry negative terms are equal and this is equivalent to showing that $\overline{R}(a+b+2,X)$ is unimodal. We consider the following cases:
\begin{itemize} \setlength\itemsep{0em}
    \item If $X$ has one part only, $a>1$ and $b=1$, then we end up with a $2$ part composition that is unimodal.
    \item Otherwise if $x_1>1$ then we have two consecutive pieces greater than $1$ as $a+b+1\geq 1$ which gives unimodality.
    \item If $x_1=1$ and $b\geq x_2$ we get a trio $a+b+1, 1,x_2$ with difference between $a+b+1$ and $x_2$ at least $2$, which gives us unimodality.
    \item Finally, if $a>1$ even if $b=x_2-1$, difference between $a+b+1$ and $x_2$ at least $2$ so we have unimodality.
\end{itemize}
\end{proof}

\begin{proof}[Proof of Theorem~\ref{unimodality}] As the oter cases are already resolved, we will focus our attention to the case of $\bar{a}$ having at least $6$ parts. By Proposition~\ref{prop:simplify} and the preceeding work, it is sufficient to show unimodality when all parts of $\bar{a}$ are $2$ or $1$. We can further suppose we have no consecutive $2,2$ or $2,1,2$ as the former is not unimodal, and the latter can be simplified. Then $\bar{a}$ either contains consecutive parts $2,1,1,2$ or $1,1,1,2$, or it consists entirely of $1$s.

If $\bar{a}$ contains consecutive parts $2,1,1,2$ then  $\bar{a}=(2,1,1,2,1,1,X)$ for some $X$ by our assumptions. As $\overline{R}((1,3,1,1,X);q)$ is unimodal, so is $\overline{R}(\bar{a};q)$.

If $\bar{a}$ contains consecutive parts $1,1,1,2$ then either $\bar{a}=(1,1,1,2,1,1)$ or$\bar{a}=(X,b,1,1,1,2,1,1)$ for some $X$ and for some $b \in \{1,2\}$ by our assumptions. The former case can be directly calculated. For the latter case, we can use Proposition~\ref{prop:simplify} with the three $1$s in the middle. As $\overline{R}((X,b+3,1,1);q)$ is unimodal, so is $\overline{R}(\bar{a};q)$.

That only leaves the case where $\bar{a}$ contains $1$'s only. Again it is easy to show $\overline{R}((1,1,1,1,1,1);q)$ is unimodal by direct calculation. Otherwise  $\bar{a}=(1,1,1,1,1,1,1,X)$ for some $X$. As $\overline{R}((3,1,1,X);q)$ is unimodal, so is $\overline{R}(\bar{a};q)$ by Proposition~\ref{prop:simplify}.
\end{proof}

\subsection{Multimodality}

Though the rank polynomials corresponding to chainlink posets are almost always unimodal, one can observe multimodality in a slightly more general setting. consider \definition{stretched chainlink polytopes}, obtained by adding the same number of parts of size $l$ between the $a_i$. In other words, the $k$-stretch of $\HL(\bar{a},l)$ is given by $\HL((a_1,l^k,a_2,l^k,\ldots,a_s,l^k),l)$. One can also define the corresponding \definition{stretched chainlink posets}. The $1$-stretched version of $\CLP((6,4,5),2)$ from Figure~\ref{fig:posetexample} is depicted  in Figure~\ref{fig:stretchedposetexample}.

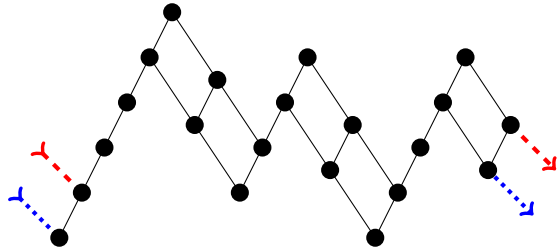
\begin{figure}
    \centering
    \begin{tikzpicture}[scale=.6]
		\node  (1) at (-16.5, 1) {};
		\node  (2) at (-16, 2) {};
		\node  (3) at (-17, 0) {};
		\node  (4) at (-15.5, 3) {};
		\node  (5) at (-15, 4) {};
		\node  (6) at (-14.5, 5) {};
		\node  (7) at (-13, 1) {};
		\node  (8) at (-12.5, 2) {};
		\node  (9) at (-12, 3) {};
		\node  (10) at (-11.5, 4) {};
		\node  (11) at (-10, 0) {};
		\node  (12) at (-9.5, 1) {};
		\node  (13) at (-9, 2) {};
		\node  (14) at (-8.5, 3) {};
		\node  (15) at (-8, 4) {};
		 \draw[->,red,dashed,ultra thick] (-7,2.5) to (-6,1.5);
		\draw[>-,red,dashed,ultra thick] (-17.5,2)--(-16.5,1);
		\draw[->,blue,dotted, ultra thick] (-7.5,1.5) to (-6.5,.5);
		\draw[>-,blue,dotted,ultra thick] (-18,1)--(-17,0);
		\fill(-16.5,1) circle(.2); 
	    \fill(-16,2) circle(.2); 
		\fill(-17,0) circle(.2); 
		\fill(-10.5, 2.5) circle(.2);
		\fill(-11, 1.5) circle(.2);
		\fill(-13.5,3.5) circle(.2);
		\fill(-14, 2.5) circle(.2);
		\fill(-15.5,3) circle(.2); 
	    \fill(-15,4) circle(.2); 
		\fill(-14.5,5) circle(.2); 
		\fill(-13,1) circle(.2); 
	    \fill(-12.5,2) circle(.2); 
		\fill(-12,3) circle(.2); 
		\fill(-11.5,4) circle(.2); 
	    \fill(-10,0) circle(.2); 
		\fill(-8,4) circle(.2); 
		\fill(-9.5,1) circle(.2); 
	    \fill(-9,2) circle(.2); 
		\fill(-8.5,3) circle(.2); 
		\fill(-7,2.5) circle(.2); 
		\fill(-7.5,1.5) circle(.2); 
		\draw (3.center) to (6.center);
		\draw (6.center) to (8.center);
		\draw (5.center) to (7.center);
		\draw (7.center) to (10.center);
		\draw (10.center) to (12.center);
		\draw (9.center) to (11.center);
		\draw (11.center) to (15.center);
		\draw (-8,4)--(-7,2.5)--(-7.5,1.5) (-8.5,3)--(-7.5,1.5) (-10.5, 2.5)--(-11,1.5) (-14,2.5)--(-13.5,3.5);
    \end{tikzpicture}
    \caption{The $1$-streched chainlink poset with $\bar{a}=(6,4,5)$ and $l=2$.}
    \label{fig:stretchedposetexample}
\end{figure}

The following corollary follows from Theorem~\ref{prop:symmetry} 

\begin{corollary} The rank polynomial of any stretched chainlink poset is symmetric.
\end{corollary}

Using this, one can extend the main theorem to show stretched chainlink polytopes also enjoy complementary sections with the same Ehrhart quasipolynomial. They behave differently in terms of unimodality though. 

In the unstretched case, $\CLP((2,2),1)$ is the smallest non-unimodular example and it has two peaks.
Let us see how stretching affects the rank sequence:
\begin{itemize} \setlength\itemsep{0em}
    \item No stretch: $[1,2,1,2,1]$.
    \item $1$-stretch: $[1,2,3,2,3,2,3,2,1]$.
    \item $2$-stretch: $[1,2,3,4,3,4,3,4,3,4,3,2,1]$.
     \item $3$-stretch: $[1,2,3,4,5,4,5,4,5,4,5,4,5,4,3,2,1]$.¨
     \item $k$-stretch: $[1,2,3,\ldots,k+2,\overbrace{k+1,k+2}^{k+1},k+1,k,\ldots,2,1]$.
     
\end{itemize}

\begin{obs} The $k$-stretch of the chainlink poset $\CLP((2,2),1)$ has $k+2$ peaks.
\end{obs}

\section{Properties of Chainlink polytopes}

In this section, we examine some properties of chainlink polytopes. 
\begin{lemma}
Let $\bar{a} \in \mathbb{N}^s$ be a composition of $n$ and let $l \in \mathbb{R}$. The chainlink polytope $\CL(\bar{a}, l)$ is full dimensional when $l < \operatorname{min} (\bar{a})$.
\end{lemma}
\begin{proof}

Let $\{\epsilon_i \}_{i \in [s]}$ be small positive real numbers all less than $\operatorname{min} (\bar{a}) - l$. Consider the point 
\[ x = (a_1 - \epsilon_1, \ldots, a_s - \epsilon_s).\] We have that 
\begin{align*}
    (a_i - \epsilon) - (a_{i (\operatorname{mod}(s)) + 1} - \epsilon) =&\,\,a_i - a_{i (\operatorname{mod}(s)) + 1} + \epsilon_i - \epsilon_{i (\operatorname{mod}(s)) + 1}\\
    <& \,\,a_i - \operatorname{min} \bar{a} + \epsilon_i \\
    <& \,\,a_i - l,
    \end{align*}
by the condition we have imposed on the $\epsilon_i$. Consequently,all points of the form $x$ above are in the polytope and constitute a full dimensional subset. 
\end{proof}
Determining exactly when these polytopes are non-empty is a tricky problem and does not seem to have a nice solution. We note though that a routine application of LP duality shows that the condition $l \leq (a_1 + \ldots + a_s)/s$ is necessary. 
\begin{lemma}
 Let $\bar{a} \in \mathbb{R}_{>0}^s$ and $ l \in \mathbb{R}_{\geq 0}$. Suppose that $ 0<l <\min_{i \in [s]} a_i$. Then the polytope $\CL ( \bar{a},l)$ has exactly $3s$ facets, defined by the equalities $ x_i =0$, $ x_i = a_i$ and $ x_i -x_{i+1} = a_i -l$. 
\end{lemma}
\begin{proof}
Let $i \in [s]$. Let $ \epsilon= ( \epsilon_1, \dots, \epsilon_s) \in \mathbb{R}^{s}_{\geq 0}$ be such that $ \epsilon_i =0$ and so that $ \epsilon_j \leq \min_{i \in [s]}a_i-l$. Then both $ \epsilon$ and $ (a_1, \dots,a_s)- \epsilon$ are in $ \CL ( \bar{a}, l)$. Thus the faces of $\CL ( \bar{a},l)$ defined by $ x_i =0$ and $ x_i =a_i$ are $s-1$-dimensional.

Take now a small positive $ \delta$ so that $\min \lbrace l, a_{i+1}-l \rbrace >\delta$. Consider the point
\begin{align*}
  p=  (\epsilon_1, \dots, \epsilon_{i-1}, a_i - l + \delta, \delta, \epsilon_{i+1}, \dots, \epsilon_s ).
\end{align*}
Then $ p \in \CL ( \bar{a}, l )$ as well. Thus the face defined by $ x_i -x_{i+1} = a_{i}-l$ is also $s-1$-dimensional. 
\end{proof}
\begin{lemma}
\label{vertexlemma} Let $\bar{a} \in \mathbb{R}_{>0}^s$ and $ l \in \mathbb{R}_{\geq 0}$. Suppose that $ 2l \leq \min_{i \in [s]} a_i$. 
The vertices $v$  of $ \CL ( \bar{a} , l )$ have the form $v_i \in \lbrace 0, l , a_i -l, a_i \rbrace $. Moreover each edge must be parallel either to the standard base vectors $e_i$ or to $e_{i}+e_{i+1} $ for some $i \in [s].$
\end{lemma}
\begin{proof}
Take $ i \in [s]$. Consider the facet $ F_i$ defined by $ x_i - x_{i+1} = a_i -l$.  Let $ p \in F_i$. If $a_i  >p_i > a_i -l$, take $p^+ = p + (a_i-p_i) (e_i - e_{i+1})$ and $ p^- = p +(a_i -l - p_i) ( e_{i}-e_{i+1})$. One sees easily that both $p^+ $ and $p^- $ are in $ \CL ( \bar{a}, l )$ and since $p$ is a convex sum of $ p^+$ and $p^-$, $p$ can not be a vertex. Thus, a vertex $ v \in F_i$ must satisfy $ v_i \in \lbrace a_i-l, a_i \rbrace$ and so $ v_{i+1} \in \lbrace 0, l \rbrace$. Since all the other facets are given by $x_i =0$ or $x_i = a_i$, we get that the coordinates of a vertex $v$ must be of the form  $v_i \in \lbrace 0, l , a_i -l, a_i \rbrace $. 

We'll show that the edges must be parallel to the $e_i$ or $e_i + e_{i+1} $ assuming that $ 2l < \min_{i \in [s]} a_i$. But since any $\bar{a}$ with $ 2l \leq \min_{i \in [s]} a_i$ can be approximated by $\bar{a}'$ satisfying the strict inequality, the statement holds in this case as well.

The kernels of the functionals that define the chainlink polytope have the following form
\begin{align*}
   \mathcal{A}_i&= \ker x_i =  span \lbrace e_j: j \in [s] -\lbrace i \rbrace \rbrace .\\
  \mathcal{B}_i &=  \ker x_i - x_{i+1} = span\lbrace e_i + e_{i+1}, e_j: j \in [s] - \lbrace i , i+1 \rbrace \rbrace.
\end{align*}
Call $A_i= \lbrace e_j: j \in [s] -\lbrace i \rbrace \rbrace$ and $ B_i =  \lbrace e_i + e_{i+1}, e_j: j \in [s] - \lbrace i , i+1 \rbrace \rbrace$.   Let $I\subset [s]$, we have
\begin{align*}
    \bigcap_{i \in I} \mathcal{A}_i = span ( \bigcap_{i \in I} A_i )
\end{align*}
 Let $J \subset [s]$ be a subset which contains no (cyclically) adjacent elements, then
 \begin{align*}
      \bigcap_{j \in J} \mathcal{B}_j = span ( \bigcap_{j \in J} B_j ).
 \end{align*}

Let $v$ be a vertex and $e$ an incident edge. Since $ 2l < \min_{i \in [s]} a_i$, the functionals $ x_{i-1} -x_i$ and $x_{i} -x_{i+1}$ can't both be maximized at $v$. Hence the set $ K=\lbrace j \in [s]: v_{j}- v_{j+1} = a_j -l \rbrace$ doesn't contain any (cyclically) adjacent elements. The edge $e$ must be parallel to a 1-dimensional subspace that is the intersection of the kernels of some of the functionals that are maximized at $v$. Thus $e$ is parallel to a 1-dimensional space of the form
\begin{align*}
    L =span(\bigcap_{i \in I} A_i) \cap span ( \bigcap_{j \in J} B_j).
\end{align*}
From this it's easy to see that $L = \mathbb{R} e_i$ or $ \mathbb{R} ( e_i + e_{i+1})$ for some $i \in [s]$.
\end{proof}

\begin{proof}[Proof of Proposition \ref{prop:vertices}]
By Lemma \ref{vertexlemma} the vertices are integral as any vertex $v$ must be of the form $ v_i \in \lbrace 0,l,a_i-l,a_i \rbrace$. 

For the second part, notice that the edges are transverse to the hyper-planes $ H^t=\lbrace x_1+ \cdots + x_s = t \rbrace$. Hence the vertices of $ \CL^t ( \bar{a},l \rbrace$ are the intersection of the edges with $H^t$. These intersections must be of the form $ v + a e_i$ or $v + b ( e_{i} + e_{i+1})$ for some vertex $v$ and some $a ,b\in \mathbb{R}.$ Since the vertices are integral, we must have $a \in \mathbb{Z}$ or $ b \in \frac{1}{2} \mathbb{Z} .$
\end{proof}
\begin{prop}
If we have the strict inequality $ 2l < \min_{i \in [s]} a_i$, then the polytope
$ \CL ( \bar{a}, l )$ is simple and the combinatorial structure doesn't depend on $\bar{a}$ or $l$.
\end{prop}
\begin{proof}
Let $v$ be a vertex. Denote by $ F( v)$ the set of defining functionals which are maximized in $ \CL (\bar{a},l) \text{ on } v $. We'll show that $v$ is simple by constructing a bijection
\begin{align*}
    f:[s] \xrightarrow[ ]{} F(v).
\end{align*}
If $v_i =0 $, map $f: i \mapsto -x_i$ and if $ v_i = a_i$  map $ f:i \mapsto x_i$. If we have $ v_{i} \not \in \lbrace 0, a_i \rbrace$, then either $v_i = a_i -l$ or $ v_i =l$. In the former case we must have $ v_{i+1}=0$ and in the latter $v_{i-1}= a_{i-1} $. In the former case map $ f:i \mapsto x_i - x_{i+1}$ and in the latter $ f:i \mapsto x_{i-1} - x_i$. Suppose that a linear functional $ \phi \in F(v) $ is not in the image of the function $f$. Clearly it can not be either of the functionals $x_i$ or $-x_i$ for any $i$. So
 $\phi= x_i-x_{i+1}$ for some $i$. Hence either $v_i =a_i$ and $v_{i+1}=l$ or $ v_{i} = a_i -l$ and $ v_{i+1}=0$. But  then in the former case $f(i+1) = x_i-x_{i+1}$ and in the latter case  $f( i)= x_{i-1} - x_i$. So $v$ is simple. 

Let $F  $ be a subset of the defining linear functionals of size $s$. Suppose that $F$ satisfies the following for each $i \in [s]$
\begin{itemize}
    \item At most one of $ x_i$ and $ -x_{i}$ is in $ F$.
    \item At most one of $x_i-x_{i+1} $ and $ x_{i-1}- x_{i}$ is in $F$.
    \item If $x_i - x_{i+1} \in F$ then either $ x_{i} \in F$ or $ -x_{i+1} \in F$.
    \item If $x_i \in F$ then $ -x_{i+1} \not \in F$. 
\end{itemize}
Then there is a unique vertex $v$ whose set of maximised functionals is $F$, that is, $ F = F(v)$.  Since the possible sets don't depend on $ \bar{a}$ or $l$, we get a combinatorial equivalence between any two $s$-dimensional chainlink polytopes satisfying $ 2l < \min_{i \in [s]} a_i$. 
\end{proof}

\begin{prop}
Let $\bar{a} \in \mathbb{R}_{>0}^s$ and $ l \in \mathbb{R}_{\geq 0}$. Suppose that $ 2l \leq \min_{i \in [s]} a_i$. The number of vertices of the chainlink polytope $ \CL ( \bar{a}, l) $ is given by 
\begin{align*}
   \textbf{Vert} ( \CL ( \bar{a}, l))= \operatorname{tr} ( A_1 \cdots A_s)
\end{align*}
where each $ A_i= A$ if $ a_i > 2l$ or $A_i =B$ if $ a_i = 2l$, where 
\begin{align*}
    A= \begin{bmatrix} 1 & 1 &1 \\
    1 & 0 & 0 \\
    1 & 1 & 1
    \end{bmatrix}, \quad  B = \begin{bmatrix} 1 & 1 &1 \\
    1 & 0 & 0 \\
    1 & 0 & 1
    \end{bmatrix} .
\end{align*}
\end{prop}
\begin{proof}
Let $v$ be a vertex of $ \CL ( \bar{a}, l)$. From \ref{vertexlemma}, we know that $ v_i \in \lbrace 0,l,a_i-l, a_i \rbrace$. Moreover, if $ v_i \not \in \lbrace 0, a_i \rbrace$, then $v$ is either contained in the facet defined by $ x_i - x_{i+1} = a_i -l$  and thus $ v_i = a_i -l$ and $ v_{i+1} = 0$ or $v$ is contained in the facet defined by $ x_{i-1}- x_i = a_{i-1} -l$, in which case $ v_{i-1} = a_{i-1}$ and $ v_i = l$. 

In light of this, we encode the vertices as follows: 
\begin{itemize}
    \item If $ v_i =0$ or if $ v_{i_1} =a_{i-1}$ and $ v_i =l$, the $i$th index is called \textbf{small}.
    \item If $v_{i}= a_{i}- l$ and if $ v_{i-1} - v_i \neq a_{i-1} -l$, then the $i$th index is called $\textbf{medium}$.  Notice that a medium index must be followed by a small index. 
    \item If $ v_i =a_i$, then the $i$th index is \textbf{large}.
\end{itemize}

From a vertex $v$, we construct a word $ w_v : [s]  \rightarrow \lbrace s, m,L \rbrace$ where 
\begin{align*}
    w_v (  i ) = \begin{cases} s & \text{ if  the index } i \text{ is small.} \\
    m & \text{ if  the index } i \text{ is medium.} \\
    L & \text{ if  the index } i \text{ is large.} 
    \end{cases}.
\end{align*}
The correspondence $ v \mapsto w_v$ is 1-1, given $\bar{a}$ and $w_v$, we can reconstruct $v$. The words $w_v$ obey two simple rules:
\begin{itemize}
\item An $m$ is followed by an $s$.
\item If $ a_i = 2l$ and if $ w_v (i-1) = L $, then the $i$th index can't be medium, that is $ w_v( i) \neq m$. 
\end{itemize}

Consider words of length $s+1$ that satisfy the above two rules. Construct matrices  $M(\bar{a}) = ( M_{ij}(\bar{a})_{i,j \in \lbrace m,s L \rbrace}$ so that $M_{ij}(\bar{a})$ are the number of length $ s+1$-words that begin with $i$ and end  with $j$.

Let $ \bar{a}'$ denote the first $s-1$ terms of $\bar{a}$. If the last entry $ a_{s}=2l$, we have
\begin{align*}
    M (\bar{a}) = M(\bar{a}') \begin{bmatrix} 1 & 1 &1 \\
    1 & 0 & 0 \\
    1 & 0 & 1
    \end{bmatrix}.
\end{align*}
And if $a_i > 2l$, we have
\begin{align*}
    M(\bar{a}) = M(\bar{a}') \begin{bmatrix} 1 & 1 &1 \\
    1 & 0 & 0 \\
    1 & 1 & 1
    \end{bmatrix}.
\end{align*}
Since the vertices of $ \CL ( \bar{a}, l ) $ correspond to the words that begin and end with the same letter, we get the trace formula in the proposition.
\end{proof}
\begin{corollary}
In particular if $ 2l < \min_{ i \in [s] } a_i$, then 
\begin{align*}
    \textbf{Vert} ( \CL ( \bar{a}, l)) = \operatorname{tr} (A^s)
\end{align*}
which satisfies the linear recurrence
\begin{align*}
    \operatorname{tr} (A^s)= 2 \operatorname{tr} (A^{s-1}) + \operatorname{tr} (A^{s-2})
\end{align*}
for $ s \geq 3$. It can be seen easily that $ \operatorname{tr} (A) =2$ and $ \operatorname{tr} (A^2) =6$. These are dubbed the "Companion Pell Numbers" in A002203.
\end{corollary}
\begin{corollary}
If $ a_i =2l$ for each $ i \in [s]$, then 
\begin{align*}
    \textbf{Vert} ( \CL ( \bar{a}, l)) = \operatorname{tr} (B^s)
\end{align*}
which satisfies the linear recurrence
\begin{align*}
    \operatorname{tr} (B^s)=2  \operatorname{tr} (B^{s-1}) +\operatorname{tr} (B^{s-2}) - \operatorname{tr} (B^{s-3})  
\end{align*}
for $ s \geq 3$. It's easy to see that $ \operatorname{tr} (B^{0}) =3 $, $ \operatorname{tr} (B)=2 $ and $ \operatorname{tr} (B^{2})=6 $. This is the sequence
A033304. Note that the matrix $B$ shows up but it is conjugated by a symmetric matrix. 
\end{corollary}
The calculation of the volume of a chainlink polytope has quite a straightforward formula in the case $ 2l \leq \min_{i \in [s]} a_i$. 
\begin{prop}
Let $\bar{a} \in \mathbb{R}_{>0}^s$ and $ l \in \mathbb{R}_{\geq 0}$. Suppose that $ 2l \leq \min_{i \in [s]} a_i$.  The volume of the chainlink polytope $ \HL( \bar{a},l )$ is given by the following trace formula.
 \begin{align*}
    \Vol ( \HL( \bar{a}, l ))= \operatorname{tr} \left(\begin{bmatrix}
    a_1 & \frac{-l^2}{2} \\
    1 & 0   
    \end{bmatrix}\cdot \begin{bmatrix}
    a_2 & \frac{-l^2}{2} \\
    1 & 0   
    \end{bmatrix}\cdot \begin{bmatrix}
    a_3 & \frac{-l^2}{2} \\
    1 & 0   
    \end{bmatrix}\cdots\begin{bmatrix}
    a_s & \frac{-l^2}{2} \\
    1 & 0   
    \end{bmatrix}\right).
\end{align*}

\end{prop}
\begin{proof} We shall give a description of the volume of the polytope in terms of the mathings of the cyclic graph on $[s]$. We define a matching to be any subset of edges that are pairwise disjoint, and denote by $\mathcal{M}_k ([s])$ matchings of the cyclic graph on $[s]$ with exactly $k$ edges. We will additionally use the shorthand $i \in M$ to denote when $i \in [s]$ is covered by an edge of a matching $M$.
The chainlink polytope is the rectangular prism $ P=\prod_{i=1}^s [0,a_i]$ with the sets  $ S_{(i,i+1)}:i \in [s]$ shaved off, where
\begin{align*}
    S_{(i,i+1)} = \lbrace x \in P: x_i -x_{i+1} > a_i -l \rbrace.
\end{align*}
We'll think of the indexing of the sets as edges of the cyclic graph on $1, \dots,s$.
By inclusion-exclusion, one gets
\begin{align*}
   \Vol ( \HL( \bar{a}, l ))= a_1 \cdots a_s+ \sum_{k=1}^s \sum_{1 \leq i_1 < i_2< \dots <i_k \leq s} \Vol( \bigcap_{j=1}^k S_{(i_j,i_{j+1})} ) (-1)^k.
\end{align*}
If two edges $e$ and $f$ of the cyclic graph $C_s$ intersect, then the intersection $S_e \cap S_f = \emptyset$. Thus we only need to be concerned with the terms that come from matchings $m \in \mathcal{M}_k (s)$ in the above formula. For a matching $M  \in \mathcal{M}_k (s) $, we have, after permuting the coordinates,
\begin{align*}
    \bigcap_{e \in M} S_e = \prod_{i \not \in M} [0,a_i] \times \prod_{(i,i+1) \in M} \lbrace (x,y) \in \mathbb{R}^{2} : a_{i } \geq x > a_{i} -l , l > y \geq 0, x - y > a_{i} -l \rbrace
\end{align*}
Here, we consider $(i,i+1)$ modulo $s$ as usual. The volume of the above is clearly $ \prod_{i \not \in m} a_i \frac{l^{2k}}{2^k}$, which matches the trace formula.
\end{proof}

\section{Remarks and further work}
There are several questions about these chainlink polytopes that naturally arise. 
\begin{itemize}
\item \textbf{Ehrhart-Equivalence:} Two  rational polytopes $P, Q \in \mathbb{R}^d$ are said to be Ehrhart-Equivalent if they have the same  Ehrhart quasi-polynomial.  They are said to be $GL$ equidecomposable if we may partition $P = U_1 \cup \ldots \cup U_n$ and $Q = V_1 \cup \ldots \cup V_n$ into relatively open simplices such that for each $i$, $U_i$ and $V_i$ are $GL_d(\mathbb{Z})$ equivalent. In \cite{haase2008quasi}, it was conjectured that Ehrhart-equivalent polytopes are $GL$ equidecomposable. This is known to be true for dimensions $2$ \cite{greenberg1993piecewise} and $3$ \cite{erbe2019ehrhart}. Sections of chainlink polytopes provide us with a large class of examples to test this conjecture. 
\item \textbf{Multimodality:} Theorem \ref{unimodality} can be expressed in the following way: Let $\bar{a}$ be a composition of $n$. Then the function from $\{0, \ldots, n\}$ to $\mathbb{N}$ given by 
\[k \rightarrow \# CL^k(\bar{a}, 1),\]
is unimodal save when $\bar{a} = (a, 1, a, 1)$ or $(1, a, 1, a)$  and is bimodal in these cases. If we instead fix a positive integer $l$ such that $2l \leq \operatorname{min} \{a_i\}$ and look at 
\[k \rightarrow \# CL^k(\bar{a}, l),\]
the function may be multimodal. Indeed, we have that when $\bar{a} = (2k, 2k)$ and $l = k$, we seem to have $k+1$ peaks. Can one describe the maximal number of modes that may arise for fixed $l$ and when these are attained? 

\item \textbf{The General Chainlink Polytope:} In the case $ 2l > \min_{i \in [s]} a_i$, most interesting properties of the polytope $ \CL ( \bar{a},l)$ vanish. Namely, the vertices of the sections are no longer half-integral, we lose the  equality of volumes of complementary sections, various formulae for the number of the vertices and the volume of $\CL ( \bar{a}, l)$ no longer hold. Given that the chainlink polytope $ \CL ( \bar{a}, l)$ is full-dimensional when $ l < \min_{i \in [s]} a_i$, this leaves a lot to be investigated, both combinatorially and geometrically.
\end{itemize}

\bibliographystyle{plain}
\bibliography{chainlink}
\end{document}